\documentclass[11pt]{amsart}
\usepackage{amssymb,amsmath,amsthm,amsfonts,mathrsfs}
\usepackage[hmargin=3cm,vmargin=3.5cm]{geometry}
\usepackage[dvipsnames,table,xcdraw]{xcolor}
\usepackage[colorlinks = true,
            linkcolor = Fuchsia,
            urlcolor  = ForestGreen,
            citecolor = WildStrawberry,
            anchorcolor = blue]{hyperref}
\usepackage{graphicx,psfrag}
\usepackage{amscd}
\usepackage[shellescape]{gmp}
\usepackage{stmaryrd}  
\usepackage[all,2cell]{xy} \UseAllTwocells \SilentMatrices
\usepackage{tikz-cd}
\usepackage[dvips]{epsfig}
\usepackage{MnSymbol}
\usepackage{comment}
\usepackage{enumitem}
\usepackage{kbordermatrix}
\usepackage[colorinlistoftodos]{todonotes}
\usepackage{cancel}

\renewcommand{\kbldelim}{(}
\renewcommand{\kbrdelim}{)}

\usetikzlibrary{arrows,automata}
\usetikzlibrary{decorations.markings}
\usetikzlibrary{decorations.pathmorphing}

\makeatletter
\def\@adminfootnotes{%
  \let\@makefnmark\relax  \let\@thefnmark\relax
  \ifx\@empty\@date\else \@footnotetext{\@setdate}\fi
  \ifx\@empty\@subjclass\else \@footnotetext{\@setsubjclass}\fi
  \ifx\@empty\@keywords\else \@footnotetext{\@setkeywords}\fi
  \ifx\@empty\thankses\else \@footnotetext{%
    \def\par{\let\par\@par}\@setthanks}%
  \fi
}
\makeatother

\newtheorem{theorem}{Theorem}[section]
\newtheorem{proposition}[theorem]{Proposition}
\newtheorem{prop}[theorem]{Proposition}

\newtheorem{corollary}[theorem]{Corollary}

\theoremstyle{definition}

\newtheorem{example}[theorem]{Example}

\theoremstyle{remark}
\newtheorem{remark}[theorem]{Remark}

\numberwithin{equation}{section}

\let\oldtocsection=\tocsection
\let\oldtocsubsection=\tocsubsection
\renewcommand{\tocsection}[2]{\hspace{0em}\oldtocsection{#1}{#2}}
\renewcommand{\tocsubsection}[2]{\hspace{1em}\oldtocsubsection{#1}{#2}}

\usepackage{chngcntr}
\counterwithin{figure}{subsection}

\makeatletter
\@namedef{subjclassname@2020}{%
  \textup{2020} Mathematics Subject Classification}

\makeatother
 
\title[Entropy, cocycles, and their diagrammatics]{Entropy, cocycles, and their diagrammatics}

\author{Mee Seong Im}
\address{\parbox{\linewidth}{Department of Mathematics, Johns Hopkins University, Baltimore, MD 21218, USA\\ 
Department of Mathematics, United States Naval Academy, Annapolis, MD 21402, USA}}
\email{\href{mailto:meeseong@jhu.edu}{meeseong@jhu.edu}}
 
\author{Mikhail Khovanov} 
\address{\parbox{\linewidth}{Department of Mathematics, Johns Hopkins University, Baltimore, MD 21218, USA 
\\ 
Department of Mathematics, Columbia University, New York, NY 10027, USA}}
\email{\href{mailto:khovanov@jhu.edu}{khovanov@jhu.edu}}

\date{October 8, 2024}

\providecommand{\keywords}[1]{\textbf{\textit{Key words and phrases.}} #1}


\begin{document}

\def\Aff{\mathsf{Aff}}
\def\AND{\mathsf{AND}}
\def\concatenate{\mathsf{concatenate}}
\def\Br{\mathsf{Br}}
\def\Gal{\mathsf{Gal}}
\def\gen{\mathsf{generators}}
\def\GL{\mathsf{GL}}
\def\SL{\mathsf{SL}}
\def\init{\mathsf{in}}
\def\t{\mathsf{t}}
\def\out{\mathsf{out}}
\def\I{\mathsf I}
\def\region{\mathsf{region}}
\def\PMI{\mathsf{PMI}}
\def\plane{\mathsf{plane}}
\def\R{\mathbb R}
\def\Q{\mathbb Q}
\def\Z{\mathbb Z}
\def\mc{\mathcal{c}}
\def\finite{\mathsf{finite}}
\def\infinite{\mathsf{infinite}}
\def\N{\mathbb N} 
\def\C{\mathbb C}
\def\sep{\mathsf{sep}}
\def\S{\mathbb S}
\def\SS{\mathbb S} 
\def\CP{\mathbb P}
\def\Ob{\mathsf{Ob}}
\def\op{\mathsf{op}}
\def\new{\mathsf{new}}
\def\old{\mathsf{old}}
\def\OR{\mathsf{OR}}
\def\AND{\mathsf{AND}}
\def\rat{\mathsf{rat}}
\def\rec{\mathsf{rec}}
\def\tail{\mathsf{tail}}
\def\coev{\mathsf{coev}}
\def\ev{\mathsf{ev}}
\def\id{\mathsf{id}}
\def\s{\mathsf{s}}
\def\S{\mathsf{S}}
\def\t{\mathsf{t}}
\def\start{\textsf{starting}}
\def\Notation{\textsf{Notation}}
\def\circleft{\raisebox{-.18ex}{\scalebox{1}[2.25]{\rotatebox[origin=c]{180}{$\curvearrowright$}}}}
\renewcommand\SS{\ensuremath{\mathbb{S}}}
\newcommand{\kllS}{\kk\llangle  S \rrangle} 
\newcommand{\kllSS}[1]{\kk\llangle  #1 \rrangle}
\newcommand{\klS}{\kk\langle S\rangle}  
\newcommand{\aver}{\mathsf{av}}  
\newcommand{\ophana}{\overline{\phantom{a}}}
\newcommand{\Bool}{\mathbb{B}}
\newcommand{\dmod}{\mathsf{-mod}}
\newcommand{\lang}{\mathsf{lang}}
\newcommand{\pfmod}{\mathsf{-pfmod}}
\newcommand{\primitive}{\mathsf{irr}}
\newcommand{\Bmod}{\Bool\mathsf{-mod}}  
\newcommand{\Bmodo}[1]{\Bool_{#1}\mathsf{-mod}}  
\newcommand{\Bfmod}{\Bool\mathsf{-fmod}} 
\newcommand{\Bfpmod}{\Bool\mathsf{-fpmod}} 
\newcommand{\Bfsmod}{\Bool\mathsf{-}\underline{\mathsf{fmod}}}  
\newcommand{\undvar}{\underline{\varepsilon}} 
\newcommand{\RLang}{\mathsf{RLang}}
\newcommand{\undotimes}{\underline{\otimes}}
\newcommand{\sigmaacirc}{\Sigma^{\ast}_{\circ}} 
\newcommand{\cl}{\mathsf{cl}}
\newcommand{\PP}{\mathcal{P}} 

\newcommand{\whA}{\widehat{A}}
\newcommand{\whC}{\widehat{C}}
\newcommand{\whM}{\widehat{M}}

\newcommand{\mcCinfalpha}{\mcC^{\infty}_{\alpha}}
\newcommand{\mathT}{\mathsf{T}}
\newcommand{\mathF}{\mathsf{F}}
\newcommand{\mcS}{\mathcal{S}}
\newcommand{\mcN}{\mathcal{N}}
\newcommand{\wmcN}{\widetilde{\mcN}}
\newcommand{\Net}{\mathsf{Net}}
\newcommand{\Catone}{\mcC_{\kk}}
\newcommand{\Cattwo}{\mcC_{\kk}^{\circ}}
\newcommand{\CatoneR}{\mcC_{\R}}
\newcommand{\CattwoR}{\mcC_{\R}^{\circ}}
\newcommand{\Catthree}{\widetilde{\mcC}_{\kk}}
\newcommand{\Catfour}{\widetilde{\mcC}^{\circ}_{\kk}}
\newcommand{\CatoneH}{{\mcC}_{H}}
\newcommand{\CattwoH}{\mcC_H^{\circ}}
\newcommand{\BFL}{\mcC_{\mathsf{Fin}}}

\let\oldemptyset\emptyset
\let\emptyset\varnothing

\newcommand{\undempty}{\underline{\emptyset}}
\newcommand{\undsigma}{\underline{\sigma}}
\newcommand{\undtau}{\underline{\tau}}
\def\basis{\mathsf{basis}}
\def\irr{\mathsf{irr}} 
\def\spanning{\mathsf{spanning}}
\def\elmt{\mathsf{elmt}}

\def\H{\mathsf{H}}
\def\I{\mathsf{I}}
\def\II{\mathsf{II}}
\def\l{\lbrace}
\def\r{\rbrace}
\def\o{\otimes}
\def\lra{\longrightarrow}
\def\Ext{\mathsf{Ext}}
\def\mf{\mathfrak} 
\def\mcC{\mathcal{C}}
\def\mcU{\mathcal{U}}
\def\mcT{\mathcal{T}}
\def\mcO{\mathcal{O}}
\def\Fr{\mathsf{Fr}}

\def\ovb{\overline{b}}
\def\tr{{\sf tr}} 
\def\str{{\sf str}} 
\def\det{{\sf det }} 
\def\tral{\tr_{\alpha}}
\def\one{\mathbf{1}}   

\def\lra{\longrightarrow}
\def\twoheadlra{\longrightarrow\hspace{-4.6mm}\longrightarrow}
\def\hooklra{\raisebox{.2ex}{$\subset$}\!\!\!\raisebox{-0.21ex}{$\longrightarrow$}}
\def\kk{\mathbf{k}}  
\def\gdim{\mathsf{gdim}}  
\def\rk{\mathsf{rk}}
\def\undep{\underline{\varepsilon}}
\def\mathM{\mathbf{M}}  

\def\CCC{\mathcal{C}} 

\def\Cob{\mathsf{Cob}} 
\def\Kar{\mathsf{Kar}}   

\newcommand{\brak}[1]{\ensuremath{\left\langle #1\right\rangle}}
\newcommand{\brakspace}[1]{\ensuremath{\left\langle \:\: #1\right\rangle}}

\newcommand{\oplusop}[1]{{\mathop{\oplus}\limits_{#1}}}
\newcommand{\ang}[1]{\langle #1 \rangle } 
\newcommand{\ppartial}[1]{\frac{\partial}{\partial #1}} 
\newcommand{\checkr}{{\bf \color{red} CHECK IT}}
\newcommand{\checkb}{{\bf \color{blue} CHECK IT}}
\newcommand{\checkk}[1]{{\bf \color{red} #1}}

\newcommand{\mcA}{{\mathcal A}}
\newcommand{\cZ}{{\mathcal Z}}
\newcommand{\sq}{$\square$}
\newcommand{\bi}{\bar \imath}
\newcommand{\bj}{\bar \jmath}
\newcommand{\FinProb}{\mathsf{FinProb}}

\newcommand{\undn}{\underline{n}}
\newcommand{\undm}{\underline{m}}
\newcommand{\undzero}{\underline{0}}
\newcommand{\undone}{\underline{1}}
\newcommand{\undtwo}{\underline{2}}

\newcommand{\cob}{\mathsf{cob}} 
\newcommand{\comp}{\mathsf{comp}} 

\newcommand{\Aut}{\mathsf{Aut}}
\newcommand{\Hom}{\mathsf{Hom}}
\newcommand{\Idem}{\mathsf{Idem}}
\newcommand{\Ind}{\mbox{Ind}}
\newcommand{\Id}{\textsf{Id}}
\newcommand{\End}{\mathsf{End}}
\newcommand{\iHom}{\underline{\mathsf{Hom}}}
\newcommand{\Bools}{\Bool^{\mathfrak{s}}}
\newcommand{\mfs}{\mathfrak{s}}
\newcommand{\blueline}{line width = 0.45mm, blue}

\newcommand{\drawing}[1]{
\begin{center}{\psfig{figure=fig/#1}}\end{center}}

\def\endomCempt{\End_{\mcC}(\emptyset_{n-1})}

\def\MS#1{\textbf{\color{NavyBlue}[MS: #1]}}
\def\MK#1{\textbf{\color{Red}[MK: #1]}}

\begin{abstract}  The first part of the paper explains how to encode a one-cocycle and a two-cocycle on a group $G$ with values in its representation by networks of planar trivalent graphs with edges labelled by elements of  $G$, elements of the representation floating in the regions, and suitable rules for manipulation of these diagrams. When the group is a semidirect product, there is a similar presentation via overlapping networks for the two subgroups involved. 

M.~Kontsevich and J.-L.~Cathelineau have shown how to interpret the entropy of a finite random variable and infinitesimal dilogarithms, including their four-term functional relations, via 2-cocycles on the group of affine symmetries of a line. 

We convert their construction into a diagrammatical calculus evaluating planar networks that describe morphisms in suitable monoidal categories. In particular, the four-term relations become equalities of networks analogous to associativity equations. The resulting monoidal categories complement existing categorical and operadic approaches to entropy. 
\end{abstract}

\maketitle
\tableofcontents


%
%

\section{Introduction}
\label{section:intro}

The concept of entropy has fascinated mathematicians over the past decades~\cite{Gromov13,Zorich11,CC09,Rioul18,Lein21_entro_div,Baez24_entropy}. The remark ``no one knows what entropy really is", attributed to von Neumann~\cite{Rioul18}, apparently still holds in the minds of at least some of us, about eighty years later. Categorical interpretations of entropy and of the foundations of probability theory are explored in~\cite{BFL11,Baez_blog_cat,Brad21,Voevodsky04,Voevodsky08,Spivak23,BB15,Vign20,Vigneaux19}. 

Defect networks with facets labelled by elements of a group $G$ appear throughout mathematical physics, see \cite{Tach20,BBFTGGPT,BuBa20,TKBB23}
and many other references.  Labeling domain walls in a topological quantum field theory (TQFT) by elements of a finite group $G$ and introducing 2-cocycles in relation to codimension two submanifolds (seams) where domain walls for $g$ and $h$ merge into a $gh$-wall is now common in condensed matter and TQFT theories. Surface $G$-networks appear as Poincar\'e duals of Moore--Segal encodings of $G$-equivariant 2D TQFTs~\cite{MS06} and they are implicit in Turaev's homotopy QFTs~\cite{Tur10}. Finite extensions of fields give other examples of 2D TQFTs with $G$-networks~\cite{IK21}. There is a large amount of literature on defects in TQFTs, see~\cite{FMT_top_symm_24} and references therein.  

In the present work we explain how to interpret the entropy of a finite random variable via a suitable monoidal category where morphisms can be represented by pairs of overlapping planar defect networks. Our approach builds on the explicit cohomological 
interpretation of entropy by  M.~Kontsevich~\cite{Kont02} and J.-L.~Cathelineau~\cite{Cath88,Cath_comple_90,Cath96,Cath07,Cath11} in terms of a suitable 2-cocycle on the group of affine symmetries of the real line. 

Entropy and its diagrammatical calculus have additive properties, as opposed to multiplicative properties of the TQFT invariants. $G$-cocycles that appear in physics typically take values in the multiplicative group $\R^{\ast}$ or $\C^{\ast}$ of the ground field, while the present paper discusses diagrammatical calculi for \emph{additive} 1-cocycles and 2-cocycles. (Exponential of the entropy, but as an element of the tropical semiring, appears in~\cite{CC09,MT14}.)  
    
\vspace{0.07in}

Let us now briefly survey the contents of the paper. 
In Section~\ref{sec_planar} we describe several variations on a diagrammatical calculus of planar networks where lines are labelled by elements of a group $G$ and in a vertex three lines can meet as long as the appropriate product of their labels or their inverses is $1\in G$. 

\begin{itemize}
\item Section~\ref{subsec_monoidal} defines $G$-flows and corresponding monoidal categories $\mcC_G^I$ and $\mcC_G$ (these two categories are isomorphic). 
\item 
In Section~\ref{subsec_monoidal_U} to a $G$-module $U$ there is assigned a monoidal category $\mcC_U$, where morphisms are represented by diagrams of $G$-networks with boundary and elements of $U$ can float in the regions of these networks. 
\item In Section~\ref{subsec_one_cocycle} to a $U$-valued one-cocycle $f:G\lra U$ we assign a  monoidal category $\mcC_{f}$. 
\item In Section~\ref{subsec_two_cocycle} to a $U$-valued normalized two-cocycle $c:G\times G\lra U$ we assign a monoidal category $\mcC_{c}$. 
\end{itemize}

The twisting from a 2-cocycle $c$ can be encoded by a boundary line that absorbs and emits morphisms in $\mcC_{c}$, and likewise for a 1-cocycle $f$. It is shown as the blue horizontal line in the figures in Section~\ref{subsec_two_cocycle}.

\vspace{0.07in}

In Section~\ref{subsec_cross_prod} we introduce diagrammatics for a cross-product group, using two types of lines to label the elements of the two subgroups and their interactions. In Section~\ref{subsec_cross_mod} we sketch related diagrammatics for crossed modules and for non-abelian cocycles and list several examples of the latter.

\vspace{0.07in} 

Section~\ref{section:entropy} reviews cohomological interpretation of entropy of a finite random variable due to Kontsevich and Cathelineau and Cathelineau's theory of infinitesimal dilogarithms. 
This interpretation uses a 2-cocycle on the group $\Aff_1(\R)$ of affine symmetries of the real line (for entropy), and on the group $\Aff_1(\kk)$ of affine transformations over a field $\kk$ for the infinitesimal dilogarithm. 

\vspace{0.07in} 

Homological interpretation of Shannon's entropy was pointed out by 
Cathelineau in \cite{Cath96} based on his earlier work~\cite{Cath88} and by Kontsevich~\cite{Kont02}.
Entropy and its functional equation \eqref{eq_H} can be interpreted and generalized via the infinitesimal dilogarithms of Cathelineau~\cite{Cath88,Cath96,Cath11} and Bloch--Esnault~\cite{BE03}. Those papers consider the analogue of the entropy functional equation over any field. 
Furthermore, detailed relations between the infinitesimal dilogarithms for a field $\kk$ and one- and two-cocycles on various groups associated to $\kk$ are worked out in~\cite{Cath11}, see also~\cite{Kont02,Garou09,Unv21,Park07} and follow-up papers for related developments. 

\vspace{0.07in}

The main goal of the present paper is to reinterpret some of Cathelineau's and Kontsevich's work via diagrammatics for suitable monoidal categories. Affine symmetry group $\Aff_1(\kk)$ is a cross-product of the additive $\kk$ and the multiplicative $\kk^{\times}$ groups of the field $\kk$, so we utilize overlapping networks for these two groups to build monoidal categories $\Catone$ and $\Cattwo$ which are then used for an interpretation of entropy (when $\kk=\R$) and of the infinitesimal dilogarithm. Diagrammatics of 2-cocycles from Section~\ref{subsec_two_cocycle} is used as well. 

Our approach is developed in Section~\ref{sec_entropy_diagramm}. We first describe diagrammatics for Cathelineau's infinitesimal dilogarithm over any field $\kk$. Then, in Section~\ref{subsec_diag_ent}, we specialize to the ground field $\R$ and the entropy functional and corresponding categories $\CatoneH$ and $\CattwoH$. 
 A subcategory of $\CatoneH$ is equivalent the Baez--Fritz--Leinster category of finite random variables~\cite{BFL11}. One can pose a rather speculative question whether the bigger categories $\CatoneH,\CattwoH$ may give some hints about the nature of probability beyond its familiar classical and quantum interpretations. 
 
\vspace{0.07in}
   
A busy reader may want to start with Sections~\ref{section:entropy} and~\ref{sec_entropy_diagramm}, devoted to entropy, infinitesimal dilogarithms and their diagrammatics, and flip back to earlier sections as needed. 

\vspace{0.1in}

{\bf Acknowledgments:} The authors are grateful to Sri Tata for pointing out papers \cite{TKBB23,BuBa20} and an interesting discussion. 
The authors are also thankful to Danny Calegari, Andrew Dudzik, Slava Krushkal, Lev Rozansky, and Joshua Sussan for insightful conversations. M.S. Im would like to acknowledge research members at the Naval Research Laboratory in Washington, DC for interesting discussions. The authors thank Yale University, Dublin Institute for Advanced Studies, Mathematisches Forschungsinstitut Oberwolfach, Hausdorff Research Institute for Mathematics, Okinawa Institute of Science and Technology Graduate University, the Simons Foundation, and R\'enyi Institute for providing a productive work environment. M.S. Im acknowledges partial support from the Naval Research Laboratory. 
M.K. would like to acknowledge partial support from NSF grant DMS-2204033 and Simons Collaboration Award 994328. 

%
%

\section{Planar networks of group-valued flows}
\label{sec_planar}
 

\subsection{Diagrammatics of \texorpdfstring{$G$}{G}-flows and a monoidal category for a group \texorpdfstring{$G$}{G}}\label{subsec_monoidal}

$\quad$

Fix a group $G$ and consider trivalent graph networks in the plane $\R^2$ where edges of graphs are co-oriented in the plane, each edge is labelled by an element of the group $G$ and at each trivalent vertex there is a flow condition shown in Figure~\ref{fig2_001}. 

\input{fig2_001}

Co-orientation is an orientation of the normal bundle to a submanifold. 
Figure~\ref{fig2_001} shows, on the left, an individual edge with a co-orientation and a $G$-label $\sigma$, and in the middle and on the right -- two possible vertices with co-orientations and labels of adjacent edges. 
One can think of each vertex as merging two compatibly co-oriented lines $\sigma$ and $\tau$ into the line $\sigma\tau$ or $\tau\sigma$, depending on where co-orientations point. Orientation of the ambient $\R^2$ is fixed. Circles, equipped with co-orientation and labels, are allowed in the networks. 

We call these \emph{type I} \emph{$G$-flow networks}  or just \emph{type I} \emph{$G$-networks}. These networks are \emph{closed}, that is, have no boundary points. Denote the set of closed type I $G$-networks by $\Net_G^I$.  

If smooth structure is important, one can make $\sigma$ and $\tau$ lines tangent at a vertex in Figure~\ref{fig2_001}, to emphasize that they are merging into a single line (or that a single line splits into two), making all three tangency rays at a vertex belong to a single tangent line, see a later Figure~\ref{fig3_023}, for instance.  

\input{fig3_020a}

\vspace{0.07in} 

Given a $G$-network $\mcN\in \Net_G^I$ in the plane and point $p$ in the complement $\R^2\setminus \mcN$, define \emph{the winding number} $\omega(p,\mcN)$ \emph{of $\mcN$ around $p$} or \emph{the index of $p$ relative to $\mcN$} as follows, see Figure~\ref{fig3_020a}. Pick a ray going off from $p$ to infinity in the plane that intersects $\mcN$ transversally in several points (none of which are vertices and flip points; see the paragraph before Remark~\ref{remark:flip_point_lollipop} for the definition of a flip point). Multiply the $G$-labels of edges at the intersection points to 
\begin{equation}\label{eq_wind_n}
\omega(p,\mcN) \ := \ \sigma_1^{\varepsilon_1}\sigma_2^{\varepsilon_2} \cdots \sigma_n^{\varepsilon_n}\in G,
\end{equation}
where $\sigma_1,\ldots, \sigma_n$ are the $G$-labels of edges along the ray, reading from the infinite point towards $p$ and $\varepsilon_i=1$ if at the $i$-th intersection point the co-orientation is towards the infinite point of the ray and $\varepsilon_i=-1$ if the co-orientation is towards $p$. 

It is clear that $\omega(p,\mcN)$ does not depend on the choice of a generic ray out of $p$ and is an invariant of $p$ and of the region of $\R^2\setminus \mcN$ where $p$ is located.
A \emph{region} of a network is a connected component of $\R^2\setminus \mcN$. Denote by $r(\mcN)$ the set of regions of $\mcN$. 
The $G$-winding number can be viewed as a function 
\begin{equation}
\omega \ : \ r(\mcN) \lra G. 
\end{equation} 
 For points $p$ \emph{outside} $\mcN$, that is, when a ray out of $p$ to infinity disjoint from $\mcN$ exists, $\omega(p,\mcN)=1\in G$. The corresponding region is called \emph{the outside} or \emph{outer region}.  

\input{fig3_011}

\input{fig3_010}

\input{fig3_009}

Consider transformations of networks shown in Figures~\ref{fig3_011},~\ref{fig3_010} and \ref{fig3_009}. 
In Figure~\ref{fig3_011} there is bijection between the regions or the networks on the left and right, and this bijection preserves winding numbers of the regions. Likewise, there are bijections between the regions of the three networks in Figure~\ref{fig3_010}, preserving winding numbers of regions. Figure~\ref{fig3_009} transformations change the number of regions (or, in the first case, can change topology of one of the regions). We also consider Figure~\ref{fig_K3} move. 

\input{fig_K3}

\vspace{0.07in}

We are interested in having various equivalence relations on these networks, including the ones for identifying networks in Figures~\ref{fig3_011}--\ref{fig_K3} as well as an equivalence for Figures~\ref{fig3_011}--\ref{fig3_009} only, without the relation in Figure~\ref{fig_K3} holding, see Section~\ref{subsec_one_cocycle}. A convenient framework for that is the \emph{universal construction of topological theories} \cite{CFW10,Kh20_univ_const_two,IKV23,IK_22_linear,GIKKL23,IK-top-automata,IK_TQFTjourney23,IK23,IZ}. In our situation the networks are planar and one should use the universal construction for monoidal categories which are not necessarily symmetric. This setup is considered in \cite{KL21,IKV23}. 

We apply the universal construction to closed $G$-networks and, for now, do the simplest case. Pick a one-element set, say $\{0\}$, and declare that any closed $G$-network evaluates to the element of that set (to $0$). Denote this evaluation 
\begin{equation}\label{eq_alpha_zero}\alpha_0 \  : \ \Net^I_G \lra \{0\}.
\end{equation}
To build a category $\mcC_G^I$ from $G$-networks and $\alpha_0$, consider a generic cross-section of a network $\mcN$ by a horizontal line. The line intersects $\mcN$ in a number of points, which inherit labels in $G$ and co-orientations from $\mcN$. We can denote a point with label $\sigma\in G$ and co-oriented pointing to the left of $\R$ by $X^+_{\sigma}$ and a $\sigma$-labelled point co-oriented to the right by $X^-_{\sigma}$. 

Then $X^+_{\sigma},X_{\sigma}^-$ are the generating objects of $\mcC_G^I$. Given a sequence $\undsigma=(\undsigma_1,\ldots, \undsigma_n)$ of pairs (element of $G$, sign), that is, $\undsigma_i = (\sigma_i,\varepsilon_i)$, 
$\sigma_i\in G$, $\varepsilon_i\in \{+,-\}$, define the object 
\begin{equation}\label{eq_object_X}
X_{\undsigma} \ := \ X_{\sigma_1}^{\varepsilon_1}\otimes X_{\sigma_2}^{\varepsilon_2}\otimes \cdots \otimes X_{\sigma_n}^{\varepsilon_n}.
\end{equation}
Define the \emph{weight} of the sequence $\undsigma$ by 
\begin{equation}
    \brak{\undsigma} \ := \ \sigma_1^{\varepsilon_1}\sigma_2^{\varepsilon_2}\cdots \sigma_n^{\varepsilon_n} \in G. 
\end{equation}
(Diagrammatically, we merge the sequence $\undsigma$ into a single left-pointed dot (object), see Figure~\ref{fig2_002}.)

\input{fig2_002}

Morphisms from $X_{\undsigma}$ to $X_{\undtau}$ are given by all networks $D$ with boundary, $\partial_0 D =X_{\undsigma} $, $ \partial_1 D =X_{\undtau} $, modulo universal relations: two networks $D_1,D_2$ are equal if and only if no matter how they are closed on the outside into closed networks $\mcN_1,\mcN_2$, the evaluations are equal, $\alpha_0(\mcN_1)=\alpha_0(\mcN_2)$, see Figure~\ref{fig_F1}.

\vspace{0.07in} 

\input{fig_F1}

Evaluation $\alpha_0$ is constant on all closed diagrams, so the above condition is vacuous. We obtain that diagrams $D_1,D_2$ with the same bottom sequence $\undsigma$ and same top sequence  $\undtau$ (diagrams from $\undsigma$ to $\undtau$) give equal morphisms in $\mcC_G^I$, and any two morphisms from $X_{\undsigma}$ to $X_{\undtau}$ are equal. 

Merges and splits in Figure~\ref{fig2_001} and their reflections about the $x$-axis give isomorphisms \begin{equation}
    X^+_{\sigma}X^+_{\tau}\cong X^+_{\sigma\tau}, \ \ \ 
    X^-_{\sigma}X^-_{\tau}\cong X^-_{\tau\sigma}.
\end{equation} 

\input{fig3_023}

A possible isomorphism $X^+_{\sigma}\cong X_{\sigma^{-1}}^-$ is shown in Figure~\ref{fig3_023} on the right. It is convenient to encode it by a special diagram with a marked point on a line where $\sigma$-line is turned into a $\sigma^{-1}$-line and the co-orientation is reversed. We call this reversal mark a \emph{flip point}.  

\begin{remark}
\label{remark:flip_point_lollipop}
In general, one needs to be careful with flip points. There are four ways to define a flip point by inserting a $1$-labelled \emph{lollipop} of a line and a circle, both labelled $1\in G$, see Figure~\ref{fig_A1}.

\input{fig_A1}

 The circle edge can merge with the $1$-labelled edge into itself in two possible ways and the lollipop of $1$-edges can be inserted on the opposite side. These four choices can be reduced to two choices using the orientation of $\R^2$, but one needs to either pick between the two or check that the resulting morphism does not depend on the choice. For our present construction it is the latter case (since the evaluation does not depend on the diagram). In the former case,  a co-orientation can be added to the flip point to indicate which diagram is selected for the definition. Equality of morphisms is determined by the universal construction, so one needs to examine whether various choices of a lollipop result in the same evaluation of the diagram. Evaluation $\alpha_0$ in \eqref{eq_alpha_zero} takes values in the one-element set, so any two morphisms with the same source object and the same target object are equal in the quotient category. In particular, choice of a lollipop for this evaluation does not matter. 
\end{remark} 

\input{fig3_015}

Two consecutive flip points along an edge can be cancelled, see 
Figure~\ref{fig3_015} on the right, which is immediate from the degeneracy of $\alpha_0$. 

We sometimes write $X_{\sigma}$ in place of $X^+_{\sigma}$, for convenience. 
From the discussion above we see that any object $X\in \Ob(\mcC_G^I)$ is isomorphic to $X_{\sigma}$, for a unique $\sigma\in G$. The object $X_{\undsigma}$ in \eqref{eq_object_X} is isomorphic to $X_{\sigma}$ for $\sigma=\sigma_1^{\varepsilon_1}\cdots \sigma_n^{\varepsilon_n}$, where $\sigma_i^-$ stands for $\sigma_i^{-1}$ and $\sigma_i^+$ for $\sigma_i$, by merging the $n$ strands into one and adding flip points when needed to match co-orientations. Homs are given by 
\[
\Hom_{\mcC_G^I}(X_\sigma,X_{\tau}) = 
\begin{cases} 
\{\id_{X_{\sigma}}\} &  \text{if} \ \ \sigma=\tau, \\
\emptyset &  \text{otherwise}.
\end{cases}
\]
That is, the identity morphism is the only endomorphism of  $X_{\sigma}$, and there are no morphisms between $X_{\sigma}$ and  $X_{\tau}$ if $\sigma\not=\tau$. 

The category $\mcC_G^I$ is monoidal, with the tensor product given by placing diagrams next to each other. In our diagrammatics the unit object $\one$ is not represented by a dot on a line and instead can be inserted anywhere on the line. The isomorphism $\one\cong X_1$ and its inverse can be given by a lollipop morphism, see Figure~\ref{fig_F3} (for $\mcC_G^I$, the choice of lollipop convention does not matter). 

\input{fig_F3}

\vspace{0.07in} 

Category $\mcC_G^I$ is pivotal~\cite{Selinger11}, with the dual  $(X_{\sigma}^+)^{\ast}:=X^-_{\sigma}$ and duality morphisms given by cups and caps, see Figure~\ref{fig_F4_1}. Alternatively, one can set $X_{\sigma}^{\ast}:=X_{\sigma^{-1}}$, with the rigidity (U-turns) morphisms for the pivotal structure given by lollipop diagrams, see Figure~\ref{fig_A2}. For the category $\mcC_G^I$, due to the triviality of the evaluation function, the choice of duality morphisms once the formula $X_{\sigma}^{\ast}:=X_{\sigma^{-1}}$ is fixed is unique, and the four maps in Figure~\ref{fig_A2} define the same morphism in $\mcC_G$. 

\vspace{0.07in}

\input{fig_F4_1}

\input{fig_A2}

We summarize our observations on the category $\mcC_G^I$. 
\begin{prop}
    Category $\mcC_G^I$ is a pivotal monoidal category with isomorphism classes of objects parameterized by elements of $G$, and hom sets between isomorphic (respectively, non-isomorphic) objects being one-element sets (respectively, the empty set). 
\end{prop}

\begin{remark}\label{rmk_comm_G}
Objects $X^+_{\sigma}X^+_{\tau}$ and $X^+_{\tau}X^+_{\sigma}$ are isomorphic if and only if $\sigma$ and $\tau$ commute. 
Category $\mcC_G^I$ is symmetric if and only if $G$ is abelian, and the symmetric structure can be described as shown in Figure~\ref{fig_F5}. 

\input{fig_F5}
\end{remark}

\input{fig_C2}

\input{fig3_016}

Next, we use flips to define trivalent vertices where the three co-orientations all go clockwise or all go counter-clockwise around a vertex, see Figure~\ref{fig_C2}. We call these \emph{type II} vertices, and the ones with co-orientations as in Figure~\ref{fig2_001} \emph{type I} vertices. One needs to check independence of this definition from the choice of an edge where to put a flip. In the present case, it is immediate, since the evaluation function for the universal construction is constant, but in more complicated cases independence requires verification.

With this definition, there is no restriction on possible triples of co-orientations at vertices. Consequently, 
at a trivalent vertex, one or more edge co-orientations may be reversed, by adding a flip point on the edge and replacing label $\sigma$ by $\sigma^{-1}$ near the vertex, see Figure~\ref{fig3_016} for a particular choice of co-orientations. As a special case of these transformations, a flip point can slide across a vertex, as shown in Figure~\ref{fig3_017}.

\input{fig3_017}

Figure~\ref{fig3_015} (on the right) and 
Figure~\ref{fig3_016} transformations allow to remove flip points completely from an edge and from a circle.  Adding a flip point near one endpoint of an edge and then removing it via the opposite endpoint allows to convert an edge without flip points to an oppositely co-oriented edge, with the label replaced by its inverse, see Figure~\ref{fig3_019}. An edge labelled $1\in G$ can be removed, possibly adding flip points at one of both of its endpoints, see Figure~\ref{fig3_019}. 
Co-orientation and the label of the circle can be reversed, see Figure~\ref{fig3_013} bottom row. 

\vspace{0.1in}

\input{fig3_019}

Define  \emph{type II $G$-networks} by allowing \emph{flip points} and vertices of arbitrary type as have just been discussed. The evaluation function $\alpha_0$ on these more general networks is still defined to be constant. Denote the monoidal category for the universal construction via $\alpha_0$ for type II networks by $\mcC_G$. This category is equivalent (even isomorphic) to $\mcC_G^I$ as a pivotal monoidal category. In particular, sets of objects of the two categories are identical: finite sequences of pairs (element of $G$, sign).

Denote the set of $G$-networks of type II by $\Net_G^{II}$ or just by $\Net_G$. 
Any type I network is also a type II network. 
 
Allowing the possibility of a different evaluation of closed networks, one can consider planar networks of either type (I or II) built from the above elementary diagrams (trivalent vertices, flip points, edges with labels and co-orientations) modulo isotopies but no other relations. The relations would then be defined given an evaluation. 

\begin{remark} Category $\mcC_G$ is equivalent to a set-theoretic version of the monoidal category in~\cite[Example 2.3.6]{EGNO15}.
The set-theoretic version has objects $\delta_g$, over elements $g\in G$, tensor product $\delta_g\otimes \delta_h = \delta_{gh}$ and $\Hom(\delta_g,\delta_h)=\delta_{g,h}$ (where $\delta_{g,h}$ is a one-element set if $g=h$ and the empty set otherwise). In our case there are more objects, but isomorphism claases of objects are parametrized by $G$.  
\end{remark}


\subsection{A monoidal category from a \texorpdfstring{$G$}{G}-module \texorpdfstring{$U$}{U}}\label{subsec_monoidal_U}

$\quad$

Suppose $U$ is a left $\Z[G]$-module. 
Consider $G$-flow networks as above, either type I or type II, and additionally allow elements of $U$, shown as dots with labels from $U$, to float in the regions of the plane. Denote the network $\mcN$ together with floating dots in the regions by $\wmcN$, see Figure~\ref{fig3_018} for an example. We call  $\wmcN$ a $U$-network or a $(G,U)$-network. Denote by $\mcN$ the underlying $G$-network, given by removing floating points and their labels. 

\input{fig3_018}

A floating point or dot $p$ with a label $x\in U$ has the $G$-winding number $\omega(p,\mcN)\in G$ relative to the network $\mcN$, see~\eqref{eq_wind_n}. Define 
\begin{equation}\label{eq_winding}
    \alpha_U(p,\mcN) \ := \ \omega(p,\mcN)\, x \in U, 
\end{equation}
that is, we apply the group element $\omega(p,\mcN)$, the $G$-winding number of $p$, to element $x$ in $U$. Element $\alpha(p,\mcN)$ depends only on the region where $p$ floats and on its label $x$. We may also write $\alpha$ in place of $\alpha_U$  if $U$ is fixed. 

For a $U$-network $\wmcN$ define the evaluation 
\begin{equation}\label{eq_def_alpha}
\alpha_U(\wmcN) \ := \ \sum_{p\in P(\wmcN)}\alpha_U(p, \mcN) \ = \ \sum_{p\in P(\wmcN)}\omega(p,\mcN)\, x_p \in U,
\end{equation} 
where $P(\wmcN)$ is the set of floating dots of $\wmcN$. The evaluation is given by summing over all floating dots $p$ their contributions, which are labels $x_p$ twisted by $G$-winding numbers of the points. An example of an evaluation can be found in the caption for Figure~\ref{fig3_018}. The evaluation $\alpha_U$ is additive, $\alpha_U(\wmcN_1\sqcup \wmcN_2)=\alpha_U(\wmcN_1)+\alpha_U(\wmcN_2)$, where $\sqcup$ denotes placing two diagrams next to each other. 

\vspace{0.07in} 

We now use evaluation $\alpha_U$ to define a monoidal category 
$\mcC_{U}=\mcC_{G,U}$ as usual in the universal construction. We use type II networks (using type I networks results in an isomorphic category). 

Objects of $\mcC_U$ are generic cross-sections of the network above, thus they are in a bijection with objects of $\mcC_G$, which are parameterized by sequences $\undsigma$, see formula \eqref{eq_object_X}. 
We define objects $X_{\sigma}$ and $X_{\undsigma}$ of $\mcC_{U}$ in the same way as for $\mcC_G$.  

Networks $\mcN\in \Net_G$ of $G$-flows without dots are a special case of networks $\Net_U=\Net_{G,U}$ of $G$-flows with $U$-labelled dots. The inclusion of sets $\Net_G\subset \Net_U$ is compatible with evaluations, since both $\alpha_0$ and $\alpha_U$ evaluate networks in $\Net_G$ to $0\in U$, where we choose the 
obvious inclusion $\{0\}\subset U$ to relate the two evaluations: 

\vspace{0.05in}

\begin{center}
\input{comm_diag_001}
\end{center}

\vspace{0.05in}

To define morphisms in $\mcC_U$, consider diagrams $D$  of  $G$-flow networks with $U$-labelled dots with boundary  $\partial_0 D=\undsigma, \partial_1 D =\undtau$ for some sequences $\undsigma,\undtau$. 
Denote corresponding objects of $\mcC_U$ by 
$X=X_{\undsigma},$ $Y=X_{\undtau}$,  see Figure~\ref{fig_F6} for an example. 
Denote the set of such diagrams by $\Net_U(X,Y)$.

\input{fig_F6}

\vspace{0.07in}

Each diagram $D\in \Net_U(X,Y)$ defines a morphism, denoted $[D]$ (alternatively, denoted $\alpha_U(D)$) from $X$ to $Y$ in $\mcC_U$.  
Two diagrams $D_1,D_2$ with the same boundary $(X,Y)$ define the same morphism, $[D_1]=[D_2]$, if for any way to close them up via some diagram $D'$ the two evaluations are equal, $\alpha_U(D_1D')=\alpha_U(D_2D')$, also see Figure~\ref{fig_F1}, picture on the right, with $\alpha_U$ replacing $\alpha_0$. 

\vspace{0.07in} 

Category $\mcC_U$ has at least the following relations. 
Relations in Figures~\ref{fig3_011}, \ref{fig3_010} and \ref{fig3_009} hold in $\mcC_U$ (replace arrows by equalities). Flip points and trivalent vertices of new types are defined as earlier,  see Figures~\ref{fig3_023} and~\ref{fig_C2} and relations in Figures~\ref{fig3_016}, \ref{fig3_017}, and~\ref{fig3_019} hold.  

\input{fig3_012}

Two dots floating in a region can be merged into a single dot, see Figure~\ref{fig3_012} top right. A $\sigma$-circle around a dot $x$ can be converted to the dot $\sigma^{\pm 1}(x)$ depending on co-orientation of the circle, see Figure~\ref{fig3_012} bottom row. 

\input{fig3_013}

Dot labelled $x$ can move through a $\sigma$-line, being twisted by $\sigma^{\pm 1}$, depending on co-orientation of the line, see Figure~\ref{fig3_013}.

\begin{prop} \label{prop_cat_U}
    Category $\mcC_U$ is a pivotal monoidal category. Any object of $\mcC_U$ is isomorphic to $X_{\sigma}$, for a unique $\sigma\in G$. The hom spaces between these objects are   
    \begin{equation}
    \Hom_{\mcC_U}(X_\sigma,X_{\tau}) =
    \begin{cases} 
  U  &  \text{if} \ \ \sigma=\tau, \\
\emptyset &  \text{otherwise}.
\end{cases}
\end{equation}
\end{prop}

For two isomorphic objects $X,Y\in \Ob(\mcC_U)$ (we can also write $X=X_{\undsigma}, Y=X_{\undtau}$ with $\brak{\undsigma}=\brak{\undtau}$) the hom set $\Hom_{\mcC_U}(X,Y)$ is naturally a $U$-torsor, that is, (abelian) group $U$ acts on it with one orbit and trivial stabilizers.  

The group $U$ acts on the hom set by taking a diagram $D$ representing an element $[D]=\alpha_U(D)\in \Hom_{\mcC_U}(X,Y)$ and placing $x\in U$ in the leftmost region of $D$ to get a diagram ${}_x D$. As one runs over all $x\in U$, elements ${}_xD$ are in a bijection with elements of $\Hom_{\mcC_U}(X,Y)$. Picking any outside diagram $D'$ that closes inward onto the pair $(X,Y)$ of isomorphic objects, closing and evaluating $\alpha_U({}_xDD')\in U$ gives a bijection $U\to U$, $x\longmapsto \alpha_U({}_xDD')$. See Figure~\ref{fig_F7}. This bijection depends on the choice of $D$ and $D'$, but bijections differ by shift by elements of $U$.  

\vspace{0.07in} 

\input{fig_F7}

Likewise, for two diagrams $D_1,D_2$ representing elements in $\Hom_{\mcC_U}(X,Y)$, the difference 
\[ \alpha_U(D_1D') - \alpha_U( D_2 D')\in U 
\]
does not depend on the choice of outside diagram $D'$ if it closes around $D_1,D_2$ on the right (with no lines in $D'$ to the left of $D_1,D_2$) and can be denoted $[D_1-D_2]\in U$.  Furthermore, if $D_1,D_2$ differ only by placement of dots, the difference can be computed by moving all dots in $D_1,D_2$ to the leftmost region, via Figure~\ref{fig3_013} top row move, merging (adding) the dots once they are in that region, and taking the difference of the  labels of the two dots. 

Alternatively, one can choose a parametrization of $\Hom_{\mcC_U}(X,Y)$ by fixing a diagram $D$ and placing the dots $x\in U$ in the rightmost region. The two parametrizations differ by action of $\sigma$ on $U$, where $X_{\sigma}\cong X\cong Y$, since diagrams ${}_{\sigma(x)}D$ and $D_x$ are equivalent, that is, $\alpha_U({}_{\sigma(x)}DD')= \alpha_U(D_xD')$ for any closure diagram $D'$.   

\begin{remark}\label{remark_no_sums}  We do not consider formal sums of diagrams. For instance, formally taking $D+D$ and pairing with $D'$ results in evaluation $2 \alpha_U(DD')$, which is usually not of the form $\alpha_U(D_1D'')$ for any diagram $D''$, varying over all $D$. It is possible to enlarge the space of morphisms and consider a larger category where morphisms are $\Z$-linear combinations of diagrams modulo skein relation derived via the evaluation $\alpha$. 
\end{remark} 

Endomorphisms $\End_{\mcC_U}(X)\cong U$ as abelian monoids, for any $X\in \Ob(\mcC_U)$, by identifying $x\in U$ with the diagram given by placing $x$ in the leftmost region of the identity diagram $\id_X$. The latter diagram consists of vertical lines, one for each term in $X$.

 \begin{remark}
  There is an isomorphism of rigid monoidal categories $\mcC_{G,0}\cong \mcC_G$, where $0$ stands for the zero $\Z[G]$-module. 
  
  If $G$ acts trivially on $U$, dots move through diagram without twisting by elements of $G$, and there are natural isomorphisms of abelian groups $\Hom_{\mcC_U}(X_{\undsigma},X_{\undtau})\cong U$, not just of $U$-torsors, when $\brak{\undsigma}=\brak{\undtau}$, that is, when the objects are isomorphic.

 A homomorphism $\phi: U\lra U'$ of $\Z[G]$-modules induces a pivotal monoidal functor $\mathcal{F}_{\phi}:\mcC_U\lra \mcC_{U'}$ 
  which is the identity on objects and takes a diagram $D$ describing a hom in $\mcC_U$ to the diagram $\mathcal{F}_\phi(D)$ where labels $x\in U$ of floating dots are replaced by $\phi(x)\in U'$. 

  The homomorphism $0\lra U$ from the trivial $G$-module to $U$ induces a faithful monoidal functor $\mcC_G\lra \mcC_U$ for any $G$-module $U$.
  \end{remark}

  \begin{remark} 
    Category $\mcC_U$ is equivalent to the category in~\cite[Example 2.3.7]{EGNO15}. Note that the category in that example has objects corresponding to elements of $G$, while for us objects are sequences of elements of $G$ with signs. 
  \end{remark}


\subsection{Twisting by a one-cocycle}\label{subsec_one_cocycle}

$\quad$

Given a group $G$ and a left $G$-module $U$, a one-cocycle $f\in \mathsf{Z}^1(G,U)$ is a function $f:G\lra U$ subject to the condition
\begin{equation}
    f(\sigma\tau)=f(\sigma)+\sigma  f(\tau), \ \ \sigma,\tau \in G. 
\end{equation}
In particular $f(1) = f(1\cdot 1)=f(1)+1 f(1)=2f(1)$, so that $f(1)=0$ and $0=f(1)=f(\sigma\sigma^{-1})=f(\sigma)+\sigma f(\sigma^{-1})$. Thus,  
\begin{equation} \label{eq_coc_prop}  
 f(1) = 0, \ \  f(\sigma)= - \sigma f(\sigma^{-1}).
\end{equation}
A one-cocycle $f$ is \emph{null-homologous} (we write $f \in \mathsf{B}^1(G,U)$) if 
\[
f (\sigma) = \sigma u - u
\]
for some $u\in U$. The first cohomology group of $G$ with coefficients in $U$ is the quotient 
\begin{equation}\label{eq_first_homol} 
 \mathsf{H}^1(G,U) \ := \ \mathsf{Z}^1(G,U)/\mathsf{B}^1(G,U).
\end{equation}

We fix a one-cocycle $f:G\lra U$ and construct a category $\mcC_{f}=\mcC_{G,U,f}$ as follows. First, consider diagrams in $\Net^I_U$ of $G$-flows of type I with $U$-labelled dots floating in the region. Any such diagram has evaluation $\alpha_U$ given by summing twisted labels $\omega(\mcN,p)(x_p)\in U$ over the dots $p$ in $\wmcN$. We shift the evaluation via cocycle $f$. Position the network $\mcN$ so that all vertices are as in Figure~\ref{fig2_001} and the projection onto the $y$-axis is generic. Each local maximum and minimum point for this projection carries co-orientation of the line at the point (either upward or downward) and the label $\sigma$ of the line. For a local extremum point $m$ pick a point $p=p(m)$ in $\R^2\setminus \mcN$ slightly above $m$. This point has index $\omega(p,\mcN)$, the winding number of $\mcN$ around $p$. 

For an upward-cooriented local maximum or minimum point $m$ define its local contribution 
\begin{equation}\alpha_f(m):=\pm \omega(p(m),\mcN)\cdot f(\sigma)  \in U,
\end{equation}
where $+$ is chosen for maximum points and $-$ for minimum points. The contribution is given by applying the winding number $\omega(p,\mcN)\in G$ to $f(\sigma)\in U$. 

\input{fig6_001}

\begin{figure}
    \centering
\begin{tikzpicture}[scale=0.6,decoration={
    markings,
    mark=at position 0.50 with {\arrow{>}}}]
\begin{scope}[shift={(0,0)}]

\draw[thick] (0.5,1.5) .. controls (0.6,3) and (3.4,3) .. (3.5,1.5);
\draw[thick] (1.2,2.8) -- (1.3,2.5);
\node at (0.5,2.55) {$\gamma$};

\draw[thick] (1,0) .. controls (1.1,1) and (2.9,1) .. (3,0);
\draw[thick] (1.25,0.9) -- (1.4,0.6);
\node at (0.5,0) {$\sigma$};

\node at (2,1.5) {$\times$};
\node at (2.45,1.4) {$p$};

\node at (5.25,3.2) {$\omega(p,\mcN)$};

\draw[thick,gray] (2,1.5) decorate [decoration={snake,amplitude=0.25mm}] { .. controls (2.1,2.5) and (3.4,3.) .. (3.5,3.5)};

\end{scope}
\end{tikzpicture}
    \caption{Contribution of a local maximum along an upward-oriented strand to the evaluation in Figure~\ref{fig6_001} table is given by $\omega(p)f(\sigma)\in U$, where $p$ is a point slightly above the maximum. The $G$-winding number $\omega(p)=\omega(p,N)$ is given by moving $p$ to infinity, along the wavy line and taking the product $\gamma^{\pm 1}$ over the intersections of the line with arcs labelled $\gamma\in G$.  }
    \label{fig6_001a}
\end{figure}
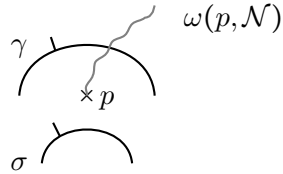

Downward-cooriented local extrema contribute $0$, so for them $\alpha_f(m)=0$. It is convenient to add vertices to the list of local diagrams and have them contribute $0$ as well. The contributions are summarized in the table in Figure~\ref{fig6_001}. Define the evaluation 
\begin{equation}
    \alpha_f(\wmcN) \ := \sum_{m} \alpha_f(m) +\alpha_U(\wmcN). 
\end{equation}
Thus, $\alpha_f$ is given by summing local contributions over all upward-cooriented local extrema of $\mcN$ and contributions from dots of $\wmcN$, see \eqref{eq_def_alpha}. 

\begin{prop}
    Evaluation $\alpha_f(\wmcN)$ is an isotopy invariant of $\wmcN$.
\end{prop}

\begin{proof} One checks that for a set of generating isotopies $\{D_1\sim D_2\}$ the contributions of $D_1, D_2$ to $\alpha_f$ in each isotopy are equal. For instance, for line isotopies in Figure~\ref{fig_C3}, diagrams $D_1,D_2$ both contribute $0$ in cases $1,4$. In case 2, the local maximum and minimum have associated dots $p_1,p_2$ above them in the same region, so that $\omega(p_1,\mcN)=\omega(p_2,\mcN)$ and the contributions from the two extrema add up to $0$ due to the minus sign, see the first two entries of Figure~\ref{fig6_001} table. The same computation takes care of case 3. 

\input{fig_C3}

\input{fig_C4}

\vspace{0.07in} 

\input{fig_F9}

For the diagrams in Figure~\ref{fig_C4}, $D_1$ has a single maximum contributing $\omega(p_1,\mcN)\cdot f(\sigma\tau)$, see Figure~\ref{fig_F9}. Diagram $D_2$ has two local maxima, with the corresponding points $p_1,p_2$ contributing 
\[
\omega(p_1, \mcN)\cdot f(\sigma) + \omega(p_2, \mcN)\cdot f(\tau), 
\]
and $\omega(p_3,\mcN)=  \omega (p_1,\mcN)\sigma$. Consequently, the contribution of $D_2$ equals 
\[
\omega(p_1,\mcN) \cdot f(\sigma) +  \omega(p_1,\mcN)\,\sigma\cdot f(\tau) = \omega(p_1,\mcN)\left( f(\sigma) + \sigma f(\tau)\right)=
\omega(p_1,\mcN)\cdot f(\sigma\tau), 
\]
where the last equality is exactly the one-cocycle condition on $f$. We see that $D_1,D_2$ contribute the same amount to $\alpha_f$. 

For Figure~\ref{fig_C4} with the opposite co-orientations of strands, both $D_1$ and $D_2$ contribute $0$ to $\alpha_f$, since in that case co-orientations point downward, see entries 3 and 4 in Figure~\ref{fig6_001}. 
Similar considerations take care of moving a vertex through a local minimum with either co-orientation. These computations also imply that the three diagrams in Figure~\ref{fig_C5} each contribute $0$ to $\alpha_f$, allowing the corresponding isotopies.  

\input{fig_C5}

\end{proof}

We can now define category $\mcC_f$ via the universal construction for evaluation 
\[
\alpha_f \ : \ \Net^I_U \lra U.
\]
For now, we are using $G$-flow networks of type $I$, without flip points or more general vertices. 
This category has the same objects $X_{\undsigma}$ as categories $\mcC_G$ and $\mcC_U$, over all finite sequences $\undsigma$, see earlier. It shares many features with $\mcC_U$, c.f. Proposition~\ref{prop_cat_U}.  

\begin{prop} \label{prop_cat_f} Given a one-cocycle $f:G\lra U$, 
    category $\mcC_f$ is a pivotal monoidal category. Any object of $\mcC_f$ is isomorphic to $X_{\sigma}$, for a unique $\sigma\in G$. The hom spaces between these objects are   
    \begin{equation}
    \Hom_{\mcC_U}(X_\sigma,X_{\tau}) =
    \begin{cases} 
  U  &  \text{if} \ \ \sigma=\tau, \\
\emptyset &  \text{otherwise}.
\end{cases}
\end{equation}
\end{prop}
For a general sequence $\undsigma$ there is an isomorphism $X_{\undsigma}\cong X_{\brak{\undsigma}}$. 
Homs between two objects $\Hom_{\mcC_f}(X_{\undsigma},X_{\undtau})$ either constitute a $U$-torsor, if $\brak{\undsigma}=\brak{\undtau}$, or the empty set, if $\brak{\undsigma}\not=\brak{\undtau}$. 
Discussion following Proposition~\ref{prop_cat_U} up until Remark~\ref{remark_no_sums} applies to $\mcC_f$ as well without any changes. 

\begin{prop}
    The following relations hold in $\mcC_f$:
    \begin{enumerate}
        \item Isotopy relations on diagrams. 
        \item Associativity of merge and split relations in Figures~\ref{fig3_011} and \ref{fig3_010} and the merge-split invertibility relations in Figure~\ref{fig3_009}.
        \item Dot relations shown in Figure~\ref{fig_F10a} (the same relations hold in $\mcC_U$).  
        \item Circle evaluation relations in Figure~\ref{fig_F10}. Innermost outward-cooriented circle evaluates to a dot labelled $f(\sigma)$. Innermost inward-cooriented circle evaluates to the  dot $f(\sigma^{-1})$. 
    \end{enumerate}
\end{prop}

\input{fig_F10a}

\vspace{0.07in}

\input{fig_F10}

Note that the relation in Figure~\ref{fig_K3} does not hold for this evaluation.

\vspace{0.07in}

The 1-cocycle relation can now be interpreted via an evaluation of nested circles in Figure~\ref{fig_F11}. 

\input{fig_F11}

\input{fig_K4} 

\begin{remark}
A closed dotless diagram evaluates to an element of $U$ which is usually not zero. For instance, the  $\sigma$-circle diagram with outward co-orientation evaluates to $f(\sigma)\in U$, see Figure~\ref{fig_F10},  and a net of three circles in Figure~\ref{fig_K4} to $f(\gamma\sigma)+\gamma f(\tau)$. To get the most freedom for that evaluation we can form the module $U$ on generators $[\sigma]$, $\sigma\in G$, with defining relations given by the one-cocycle equations $[\sigma\tau]=[\sigma]+\sigma[\tau]$, $\sigma,\tau\in G$. 
\end{remark}

We now extend the evaluation $\alpha_f$ and the diagrammatics to type II diagrams (with flip points and vertices of type II). Flip points in Figure~\ref{fig_F12} on the left (two possibilities, depending on co-orientation) are positioned vertically, and to define them as morphisms in $\mcC_f$ we bend the upper arc down to the side where the co-orientation of the lower arc points and add a lollipop, see Figure~\ref{fig_F12}.  

\input{fig_F12}

\vspace{0.07in} 

The evaluation $\alpha_f$ of a flip point is computed from this reduction of a flip point to a type I diagram. Notice that the top arc bends either to the left or to the right in Figure~\ref{fig_F12}, depending on co-orientations. The contribution to $\alpha_f$ of flip points is shown in the first two columns in Figure~\ref{fig_F13}. 

\input{fig_F13}

\vspace{0.07in} 

This normalization of flip points is chosen to have the cancellation property for two flip points on a vertical line, see Figure~\ref{fig3_015} on the right and Figure~\ref{fig_C6}. This is the normalization used in the paper. 
However, sliding flip points through local maxima and minima requires additional terms as shown in Figure~\ref{fig_F14}.  

\input{fig_C6}

\vspace{0.07in} 

\input{fig_F14}

Next, consider type II vertices, where co-orientations are compatible, all pointing clockwise or anticlockwise around the vertex, see Figure~\ref{fig_C2}, top row. We use the convention that type II vertices must be positioned so that two or one edge at the vertex point up and the other one of two edges point down. In this way the three edges at a vertex are split into two groups: two edges and one edge.  It is then natural to reduce type II vertices to type I vertices by adding a flip point to the edge which is by itself. Figure~\ref{fig_F15} shows examples of this convention. 

\vspace{0.07in} 

\input{fig_F15}

One can take a leg (or ray) at a vertex that is among the two that point together and isotop it to be next to the other one. This move has evaluation invariance only in half of the possible cases and otherwise requires an extra term, see Figure~\ref{fig_F16}. 

\input{fig_F16}

\begin{prop}
    For any one-cocycle $f\in \mathsf{Z}^1(G,U)$, categories $\mcC_f$ and $\mcC_U$ are isomorphic as monoidal categories, and an isomorphism can be made the identity on objects.  
\end{prop}

\begin{proof} Define the functor $\mathcal{F}_f:\mcC_f\lra \mcC_U$ to be the identity on objects and given as shown in  Figure~\ref{fig_F17} on morphisms. Also, the endomorphism $m_x\in \End(\one)$ of multiplication by $x\in U$ is sent to itself by $\mathcal{F}_f$.   

\input{fig_F17}

\vspace{0.07in} 

It is straightforward to see that $\mathcal{F}_f$ is an isomorphism of monoidal categories. 
\end{proof} 

Note that we are using cups and caps to capture 1-cocycle contributions. 
The significance of the twist by a cocycle $f$ is that it \emph{changes the pivotal structure of the category}, and  $\mcC_f$ and $\mcC_U$ are not isomorphic as pivotal monoidal categories, in general. The pivotal structure of the category is essential for incorporating this 1-cocycle twist to the evaluation.

\begin{remark}
    A cocycle $f$ is null-homologous if  $f(\sigma)=\sigma u -u $ for some $u\in U$. For a null-homologous cocycle maps in Figure~\ref{fig_F17} can be rewritten to not depend on the label $\sigma$ of an arc, as shown in Figure~\ref{fig_F18}. More generally, if $f_1-f_2$ is null-homologous, for cocycles $f_1,f_2$, the isomorphism between categories $\mcC_{f_1}$ and $\mcC_{f_2}$ can be defined on generators to not depend on the labels of the arcs. 

\input{fig_F18}

\end{remark}

In this section 1-cocycles of $G$ took values in an abelian group $U$. We briefly discuss diagrammatics for 1-cocycles valued in arbitrary groups in Section~\ref{subsec_cross_mod}, which requires replacing floating $U$-labelled dots by $U$-labelled planar networks interacting with those for $G$. 


\subsection{Diagrammatics for a two-cocycle}\label{subsec_two_cocycle}

$\quad$

{\bf Two-cocycles.}
Let $U$ be a $\Z[G]$-module, as before, and 
assume given a 2-cocycle $c:G\times G\lra U$, so that 
\begin{equation}\label{eq_cocondition_2}
\sigma(c(\tau,\gamma)) = c(\sigma,\tau)+ c(\sigma\tau,\gamma) -   c(\sigma,\tau\gamma), \ \ \sigma,\tau,\gamma\in G.  
\end{equation}
We can further assume the cocycle to be normalized, so that 
\begin{equation}\label{eq_unit_2}
c(\sigma,1)=c(1,\sigma)=0, \ \ 
\sigma\in G.
\end{equation} 
Denote the abelian group of normalized cocycles by 
$\mathsf{Z}^2_0(G,U)$.

Specializing \eqref{eq_cocondition_2} to $\tau=\sigma^{-1}, \gamma=\sigma$ results in the relation 
\begin{equation}\label{eq_co_symm} 
c(\sigma,\sigma^{-1}) = \sigma(c(\sigma^{-1}, \sigma)), \ \ \sigma\in G,
\end{equation}
where cocycle $c$ is normalized. 

A 2-cocycle $c$ is called \emph{null-homologous} if 
\[
c(\sigma,\tau) = b(\sigma)+\sigma(b(\tau))- b(\sigma
\tau) 
\]
for some function $b:G\lra U$. We say that a function $b$ is \emph{normalized} if $b(1)=0$. 
Denote by $\mathsf{Z}^2_0(G,U)$ the set of 2-cocycles associated to normalized functions $b:G\lra U$. These are normalized null-homologous cocycles.  

Subtracting from $c$ the null-homologous 2-cocycle for the constant function $b(\sigma)=c(1,1)$ results in a normalized 2-cocycle. We assume from now on that our 2-cocycles are normalized. 

The second cohomology group of $G$ with coefficients in $U$ can be defined as the quotient 
\begin{equation}
    \mathsf{H}^2(G,U) \ =\  \mathsf{Z}^2_0(G,U)/\mathsf{B}^2_0(G,U).
\end{equation}

Consider an extension $T$ of a group $G$ by an abelian group $U$, which can be written as 
\begin{equation}\label{eq_ext1}
1 \lra U \stackrel{\iota}{\lra} T \lra G \lra 1. 
\end{equation} 
Pick a section $s:G\lra T$, a map of sets, allowing to write $T\cong U \times G$ (isomorphism of sets) representing $t\in T$ as $u\, s(\sigma)$, parameterized by $(u,\sigma)$, for unique $u\in U, \sigma\in G$. Assume that $s(1)=1$ (normalization). Note that we write group operation in $U$ additively and multiplicatively in $T$ and $G$, so that $\iota(u_1+u_2)=\iota(u_1)\iota(u_2)$ and $\iota(0)=1$.  Then 
\[
s(\sigma)s(\tau)= c(\sigma,\tau)s(\sigma\tau),
\]
for a unique function $c: G\times G\lra U$, and $c(\sigma,\tau)$ satisfies the additive cocycle condition \eqref{eq_cocondition_2} and normalization \eqref{eq_unit_2}. This central extension has multiplication 
\[ 
(u_1,\sigma_1)(u_2,\sigma_2)=(u_1+\sigma_1(u_2)+c(\sigma_1,\sigma_2),\sigma_1 \sigma_2).
\]
Denote by $T_c$ the central extension associated to the cocycle $c$, with the short exact sequence  
\begin{equation}\label{eq_ext2}
1 \lra U \stackrel{\iota}{\lra} T_c \lra G \lra 1. 
\end{equation} 
Changing $\sigma$ by  a coboundary is equivalent to rescaling the section $s$ by a function 
\begin{equation}\label{eq_func_b}
   b:G\lra U, \ \ b(1)=0, 
\end{equation} 
so that $s'(\sigma)=b(\sigma)s(\sigma)$ and replacing $c(\sigma,\tau)$ by 
\begin{equation}\label{eq_coc_change} 
c'(\sigma,\tau) = b(\sigma)+\sigma(b(\tau))- b(\sigma\tau) + c(\sigma,\tau). 
\end{equation} 
This gives an isomorphism $T_{c'}\cong T_c$ that respects the exact sequence \eqref{eq_ext2} and its counterpart for $T_c$. 

We refer to Brown~\cite{Brown94} for more information on group cocycles and cohomology. The group $\mathsf{H}^2(G,U)$ can be interpreted as the second cohomology group of the Eilenberg--MacLane space $K(G,1)$ with coefficients in the local system given by $U$.

\input{fig_F22}

\vspace{0.07in} 

{\bf Evaluation $\alpha_c$.}
To a normalized 2-cocycle $c$ we associate an evaluation $\alpha_c$ of diagrams as shown in Figure~\ref{fig_F22}. It extends to type II vertices via their definition in  Figure~\ref{fig_F15}. 

\input{fig2_003}

\begin{prop}
\label{prop:invariance_rot_vert}
    This evaluation function is invariant under isotopies of diagrams. 
\end{prop}

\begin{proof}
We will show the invariance under rotation of a vertex in Figure~\ref{fig2_003}. 
The element of $U$ associated to the diagram on the left hand side in Figure~\ref{fig2_003} is 
\begin{equation}
\label{eq_LHS_cw}
-\omega(p_2)c(\tau^{-1},\sigma^{-1}) +\omega(p_1)c(\sigma,\sigma^{-1}) + \omega(p_3)c(\tau,\tau^{-1}) - \omega(p_1)c(\sigma\tau, (\sigma\tau)^{-1}).
\end{equation} 
We also have 
\begin{equation}
\label{eq_pushing_floating_pts}
\omega(p_2)=\omega(p_1)\sigma\tau 
\hspace{0.5cm} 
\mbox{ and } 
\hspace{0.5cm} 
\omega(p_3)=\omega(p_1)\sigma. 
\end{equation}
Rewrite \eqref{eq_pushing_floating_pts} 
as 
\begin{equation}
\label{eq_sigmatau}
\omega(p_2)c(\tau^{-1},\sigma^{-1}) = 
\sigma\tau\omega(p_1)c(\tau^{-1},\sigma^{-1}) 
\end{equation}
and 
\begin{equation}
\label{eq_sigma}
\omega(p_3)c(\tau,\tau^{-1}) = 
\sigma \omega(p_1)c(\tau,\tau^{-1}).
\end{equation}
Applying the 2-cocycle condition~\eqref{eq_cocondition_2} to the right hand side of \eqref{eq_sigmatau}, we have 
\begin{equation}
\label{eq_2cocycle_sigmatau}
\sigma\tau c(\tau^{-1},\sigma^{-1})
= c(\sigma\tau,\tau^{-1}) + c(\sigma\tau\tau^{-1},\sigma^{-1}) - c(\sigma\tau,\tau^{-1}\sigma^{-1}).
\end{equation}
Similarly, we apply 
\eqref{eq_cocondition_2} 
to the right hand side of \eqref{eq_sigma} to obtain 
\begin{equation}
\label{eq_2cocyle_sigma}
\sigma c(\tau,\tau^{-1}) = 
c(\sigma,\tau) + c(\sigma\tau,\tau^{-1})
- c(\sigma,\tau\tau^{-1}). 
\end{equation}
Combine \eqref{eq_2cocycle_sigmatau} and \eqref{eq_2cocyle_sigma} into 
\eqref{eq_LHS_cw} to obtain 
\begin{align*}
-&\cancel{\omega(p_1) c(\sigma\tau,\tau^{-1})} 
-\cancel{\omega(p_1) c(\sigma,\sigma^{-1})} 
+ \cancel{\omega(p_1) c(\sigma\tau,\tau^{-1}\sigma^{-1})}
+ \cancel{\omega(p_1) c(\sigma,\sigma^{-1})} \\
&+ \omega(p_1) c(\sigma,\tau) 
+ \cancel{\omega(p_1) c(\sigma\tau,\tau^{-1})}
- \cancel{\omega(p_1) c(\sigma,1)}
- \cancel{\omega(p_1)c(\sigma\tau, (\sigma\tau)^{-1})} \\
&= \omega(p_1) c(\sigma,\tau)
\end{align*}
since $c(\sigma,1)=0$. 
We therefore conclude 
\[
-\omega(p_2)c(\tau^{-1},\sigma^{-1}) +\omega(p_1)c(\sigma,\sigma^{-1})+\omega(p_3)c(\tau,\tau^{-1})-\omega(p_1)c(\sigma\tau, (\sigma\tau)^{-1})= \omega(p_1)c(\sigma,\tau). 
\]
For the counterclockwise rotation in Figure~\ref{fig2_003}, the proof that the diagrams in the middle and the right hand side 
of Figure~\ref{fig2_003} have the same evaluation is similar. 
This proves the compatibility under rotations of vertices.

\input{fig2_004}

Now, consider Figure~\ref{fig2_004}. 
To evaluate the left hand side, introduce a flip point at the bottom so that the co-orientations about the vertex are all pointing in the same direction. 
Then move the flip point through the vertex, resulting in two flip points along the two edges above the vertex. Using Figure~\ref{fig_F22}, the evaluation in the middle of Figure~\ref{fig2_004} is $-\omega(p)c(\sigma,\tau)$ while the evaluation on the right is $-\omega(p)c((\sigma^{-1})^{-1},(\tau^{-1})^{-1})$. Since 
$-\omega(p)c(\sigma,\tau) = -\omega(p)c((\sigma^{-1})^{-1},(\tau^{-1})^{-1})$, this case is complete.

\input{fig2_005}

We will now check the invariance of the evaluation function when moving a flip point across a local extremum, see Figure~\ref{fig2_005}. 
Using Figure~\ref{fig_F22}, the left hand side of Figure~\ref{fig2_005} evaluates to $\omega(p_1)c(\sigma,\sigma^{-1})$ while the right hand side evaluates to $\omega(p_2)c(\sigma^{-1},\sigma)$. 
Sliding $p_2$ across the $\sigma$ edge gives $\omega(p_2)=\omega(p_1)\sigma$. 
This gives 
$\omega(p_2)c(\sigma^{-1},\sigma) = \omega(p_1)\sigma c(\sigma^{-1},\sigma)$.
Since $c(\sigma,\sigma^{-1}) = \sigma c(\sigma^{-1},\sigma)$ by \eqref{eq_co_symm}, we are done.  

Remaining details are left to the reader. 
\end{proof}

{\bf Boundary wall.}
One way to interpret local evaluation rules in Figure~\ref{fig_F22} and similar rules for a one-cocycle in Figure~\ref{fig6_001} and \ref{fig_F13} is by introducing a boundary wall that can absorb and emit parts of a diagram at the cost of producing dots. 
Namely, we add the following rules for modifying networks at the boundary (and discuss only the two-cocycle case here). 

We introduce relations in Figure~\ref{fig3_026}. The first relation says that the wall (the boundary of the half-plane, shown in blue) can absorb a $(\sigma,\tau)$ vertex at the cost of adding a dot $c(\sigma,\tau)\in U$ floating to the left of the $\sigma\tau$ line. The second relation absorbs an interval $\sigma\tau$ and a vertex in the opposite way, at the cost of $-c(\sigma,\tau)$ floating in the left region. These two relations interpret the first two entries in the Figure~\ref{fig_F22} table via the absorbing wall.  

\input{fig3_026}

The two relations are equivalent modulo the relation in Figure~\ref{fig3_009} on the right, as explained in Figure~\ref{fig3_027}. 
 
\input{fig3_027}

Figure~\ref{fig3_032} shows that a version of Figure~\ref{fig3_011} associativity relation at the wall is equivalent to the 2-cocycle condition \eqref{eq_cocondition_2} on $c$. 
Note that a dot labelled $c(\tau,\gamma)$ gets twisted by $\sigma$ when it crosses the $\sigma$-line. 

\vspace{0.07in} 

\input{fig3_032}

We add the relation that a flip point can be absorbed or emitted from the fall, at no cost, see Figure~\ref{fig3_028} on the left and the corresponding entry in Figure~\ref{fig_F22}. The relation on the right of  Figure~\ref{fig3_028} follows as shown in Figure~\ref{fig3_029}. 

\input{fig3_028}

\input{fig3_029}

\begin{example}
    Figure~\ref{fig3_030} shows two ways to simplify a pair of parallel arcs at the wall labelled $\sigma$ and $\sigma^{-1}$. Naturally, the resulting elements of $U$ are equal, due to the relation $c(\sigma,\sigma^{-1}) = \sigma(c(\sigma^{-1}, \sigma))$, which is \eqref{eq_co_symm}. 

\input{fig3_030}

\end{example}

In Figure~\ref{fig3_026} on the left 
the triangle with sides $\sigma$ and $\tau$ lines and a boundary interval as the base is simplified to a dot $c(\sigma,\tau)$ in the leftmost region and a vertical $\sigma\tau$ line. Alternatively, it can be simplified in two other ways, using the Figure~\ref{fig3_026} relation on the right to absorb either the $\sigma$ or the $\tau$ line and then reduce the arc with two endpoints on the boundary line. In Figure~\ref{fig3_031} we check that the simplification along the $\tau$-line yields the same answer, and a similar computation gives the same answer for the reduction along the $\sigma$-line. This is part of checking  that the rules in Figure~\ref{fig_F22} are consistent. 

\input{fig3_031}

\begin{proposition}\label{prop_closed_diagram}
      Any closed diagram without dots evaluates to $0$.
\end{proposition}

\begin{proof} To check that any closed diagram $D$ evaluates to $0$, pick a bounded region $r$ of  $D$ with $k$ sides. 
 Cancel the co-orientation reversal dots on the sides in pairs so that each side has at most one dot. The idea is to use associativity relations in Figures~\ref{fig3_011} and~\ref{fig3_010} and their variations to reduce the number of sides of $r$ without changing the evaluation.  
 Note that two diagrams $D_1,D_2$ that differ by an edge flip as in Figure~\ref{fig3_011} on the left (vertical equality) have equal evaluation, due to the 2-cocycle property. 

For instance, evaluation invariance under the move on the left in Figure~\ref{fig3_010} is the relation
\[
c(\sigma,\tau)-c(\gamma,\gamma^{-1}\sigma\tau)= -\sigma(c(\sigma^{-1}\gamma,\sigma^{-1}\gamma\tau))+c(\sigma,\sigma^{-1}\gamma), 
\]
which is equivalent to the 2-cocycle property \eqref{eq_cocondition_2}. The move on the right in Figure~\ref{fig3_010} is the relation 
\[
-\sigma (c(\sigma^{-1}\gamma,\gamma^{-1}\sigma\tau))+c(\sigma,\sigma^{-1}\gamma) = - c(\gamma,\gamma^{-1}\sigma)+\gamma c(\gamma^{-1}\sigma,\tau)). 
\]
This relation follows by expanding the two terms with the action of $\sigma$ and $\gamma$ into signed sums of cocycles via \eqref{eq_cocondition_2} and canceling the resulting terms. 

Using these two relations and moving the flip points on edges of $r$ through vertices, if needed, we can keep shrinking the number of edges of $r$ until it has only one edge left, as shown in Figure~\ref{fig8_011} on the left.  The remaining edge is a circle labelled by $\tau\in G$. This circle and the attached $1$-edge can be deleted from the diagram since $c(\tau,1)=1$. The opposite endpoint of the $1$-edge contributes $c(\sigma,1)=0$ for some $\sigma$, allowing to delete this configuration without changing the evaluation. 

\vspace{0.07in}

\input{fig8_011}

An example of dealing with dots is shown in Figure~\ref{fig8_011} on the right, where in a digon region a flip point is present on the boundary. We bend the diagram to evaluate it, resulting in $\sigma\tau(c(\tau^{-1},\tau))-c(\sigma,\tau)-c(\sigma\tau,\tau^{-1})=0$, matching the evaluation of the vertical line $\sigma\tau$. 

This chain of transformations reduces the number of regions of $D$ by one without introducing any dots and preserving the evaluation. 
Note also that a circle evaluates to $0$ and flip points contribute $0$ to the evaluation, see Figure~\ref{fig3_028}. 
These facts together imply that any closed diagram evaluates to $0$.
\end{proof} 

\begin{remark}
    The 2-cocycle evaluation $\alpha_c$ has a different flavor from the 1-cocycle evaluation in Section~\ref{subsec_one_cocycle}, for there a closed diagram without dots can evaluate to a nonzero element of $U$. 
\end{remark}

\begin{remark}
Proposition~\ref{prop_closed_diagram} implies that, given two dotless diagrams $D_1, D_2$ with the same top and bottom boundaries, having the bottom (blue) line swallow them results in the same value $x\in U$ for the dot in the terminal diagram that consists of vertical lines and a single floating dot in the leftmost region. Consequently, the evaluation $x$ of a dotless diagram $D$ with boundary depends only on the boundary of $D$. In this way, the function that to a dotless diagram $D$assigns its evaluation $x$ can be viewed as a one-dimensional theory, depending only on the sequences of labels and co-orientations at the top and bottom lines of $D$. 
\end{remark} 

\begin{corollary}
     The category $\mcC_c$ associated to a normalized cocycle $c$ is isomorphic to $\mcC_U$ as a monoidal category.
\end{corollary}
\begin{proof}
This is immediate due to Proposition~\ref{prop_closed_diagram}, since the evaluation of a closed dotless diagram is trivial, and dots are transformed in the same way in both twisted and untwisted diagrammatics. The significance of $c$ is in a different evaluation for diagrams with boundary. 
\end{proof}

\begin{remark}\label{rmk_automorphism}
    A 2-cocycle $c$ as above gives a monoidal automorphism of $\mcC_U$, where each generating morphism is twisted by a dot placed in the region marked by $p$ in Figure~\ref{fig_F22}, and the dot is labelled by the coefficient given in the table in that figure (for instance, coefficient $-c(\sigma,\tau)$ for the second entry in the first row of the table).   
\end{remark}

\begin{remark} \label{rmk_central_ext} 
Elements of the central extension $T_c$ of $G$ by $U$ in \eqref{eq_ext2} have an  interpretation as shown in Figure~\ref{fig_F23}, where an element is represented by a vertical line $\sigma$ and a dot $x\in U$. The multiplication of two such diagrams is given by merging the lines and reducing the diagram to a dot and a vertical line, as shown in the same figure. Elements of the central normal subgroup $U\subset G$ correspond to the diagrams with $\sigma=1$, and the vertical line then can be erased, leaving us with a dot labelled $x\in U$.  
\end{remark} 

\input{fig_F23}

{\bf Changing by a coboundary.}
Given a rescaling $b$ as in \eqref{eq_func_b}, consider the locally defined map 
\[
r_b: \Net \lra \Net
\]
given by scaling each boundary point of the diagram by placing $b(\sigma)$ to the region on the left if the corresponding edge is co-oriented to the left, and $b(\sigma^{-1})$ if the edge is co-oriented to the right, see Figures~\ref{fig3_033} and \ref{fig3_034}, where new lines (after rescaling) are shown as  spirals. 

For a network of spiral lines, each endpoint at the boundary is rescaled as above. 
 In particular, if an interval enters at both endpoints, both  endpoint are rescaled. The rest of the diagram is left untouched (intervals and circles disjoint from the boundary and dots do not contribute to the scaling). 

 \input{fig3_033}
 
 \input{fig3_034}

 \vspace{0.07in} 

 We show in Figures~\ref{fig3_033} and \ref{fig3_034} how the cocycle changes after this twisting, matching the formula \eqref{eq_coc_change}. 

\begin{prop} There is a commutative diagram with horizontal arrow isomorphisms -- the rescaling $r_b$ and the isomorphism $T_{c'}\cong T_c$ described earlier and evaluation maps as vertical arrows: 
\[\begin{tikzcd}
\Net \arrow{r}{r_b} \arrow[swap]{d}{\brakspace{}_{c'}} & \Net \arrow{d}{\brakspace{}_c} \\
T_{c'} \arrow{r}{\cong} & T_c
\end{tikzcd}
\]
\end{prop} 

Thus, to each element of $\mathsf{H}^2(G,U)$ there is assigned an isomorphism class of diagrammatic calculus of $G$-networks in the half-plane.  

\begin{remark}\label{rmk_two_param}
    Given a 1-cocycle $f\in \mathsf{Z}^1(G,U)$ and a 2-cocycle $c\in \mathsf{Z}^2(G,U)$, the evaluations described in Sections~\ref{subsec_one_cocycle} and~\ref{subsec_two_cocycle} can be added and their sum used to define a monoidal category $\mcC_{c,f}$ that generalizes both $\mcC_c$ and $\mcC_f$. 
\end{remark}

\begin{remark} This and the previous sections explain how one-cocycles and normalized two-cocycles allow to twist the graphical calculus for the category $\mcC_U$. Zero-cocycles can be added as well. They correspond to $G$-invariant elements $z\in U^G$. One can add $z$ to an evaluation of a closed diagram for a one-cocycle, by placing $z$-labelled dot anywhere in the diagram. This does not change the skein relations for the category $\mcC_f$, for a one-cocycle $f$, but it kind of allows to involve $0$-cocycles for $(G,U)$ in the construction, albeit in a trivial way. 
\end{remark}

\begin{remark}
    Labeling domain walls in TQFTs by elements of a finite group $G$ and introducing 2-cocycles in relation to codimension two submanifolds (seams) where domain walls for $g$ and $h$ merge into a $gh$-wall is common in condensed matter and  physics TQFT literature, see for instance~\cite[Figure 25]{BCHKTZ23} and~\cite{Tach20}. In the physics literature the action of $G$ on $U$ is usually trivial, and floating labelled dots in the regions can then be taken out of the diagram as scalars. Another difference is that in the present paper the action on $U$ is additive, while in the physics literature the action is multiplicative, where one typically considers $\mathsf{H}^n(G,\C^{\ast})$ or a similar group.  
    Nontrivial multiplicative actions of a group $G$ on a ring $R$ appear in the theory of the Brauer group, and we hope to discuss those in the framework of suitable monoidal categories and their diagrammatics in a follow-up paper. 
\end{remark}

\begin{remark}
$G$-networks in the plane, of type I, with or without dots, can be given a topological interpretation. 
Consider the classifying space $BG$ of a group $G$. It has a cell decomposition with a single 0-cell $x_0$, edges labelled by elements of $G$, 2-simplices labelled by ordered pairs of elements of $G$, and so on. 
A planar $G$-network $\mcN$ in a bounded region $X=[a,b]\times [0,1]$ of the strip $\R\times [0,1]$ of a plane gives rise to a simplicial map $X\lra BG$. 
This map is constructed by forming a triangulation of $X$ Poincare dual to the network $\mcN$. Figure~\ref{fig8_012} shows the dual simplex to a vertex of a network. 

\vspace{0.07in}

\input{fig8_012}

The triangulation is decorated: edges are oriented and labelled by elements of $G$, so that a 2-simplex carries elements $\sigma,\tau,\sigma\tau\in G$ on its edges. Such a triangulation defines a simplicial map $\phi_{\mcN}:X\lra BG$. 

If network $\mcN$ is closed (has no boundary points), $\partial X$ is mapped to the base point of $BG$. The map $\phi_{\mcN}$ can then be viewed as a simplicial and basepoint-preserving map $\phi_{\mcN}:\SS^2\lra BG$  of a 2-sphere into the classifying space. 

Since the latter is a $K(G,1)$-space, map $\phi_{\mcN}$ is null-homotopic, through a CW-homotopy. The latter homotopy $\phi^t_{\mcN}:\SS^2\times [0,1]\lra BG$ contracts $\phi_{\mcN}(\SS^2)$ to a point and can be made combinatorial. The corresponding transformations of the network should match relations including those in Figures~\ref{fig3_011} through~\ref{fig_K3}. 

\vspace{0.07in}

Suppose given a representation $U$ of $\Z[G]$. Consider a connected CW-complex $Y$ with 
\[
\pi_1(Y)\cong G, \ \pi_2(Y)=U, \ 
\]
such that $\pi_1(Y)$ acts on $\pi_2(Y)$ via the representation $U$. A detailed analysis of related spaces can be found in~\cite{MW50}. Form a fibration $p:Y\lra BG$ with fibers homotopy equivalent to a space $W$ with $\pi_1(W)=1$ and $\pi_2(W)\cong U$. Monodromy action of $G\cong \pi_1(Y)\cong \pi_1(BG)$ on $W_0:=p^{-1}(x_0)$, the fiber above the basepoint $x_0\in BG$, matches the action of $G$ on $U$. 

Our setup also requires existence of a section $\iota:BG\lra Y$ so that $p\iota=\id_{BG}$. Pick the base point  $w_0=W_0\cap  \iota(BG)$ above $x_0$.

Given a diagram $D$ of a type I $G$-network with floating dots labelled by elements of $U$, in a patch $X$ of the plane $\R^2$, we can construct a map $X\lra Y$ corresponding to $D$. For a dot $d\in U$ labelled $a\in U$, take a tubular neighbourhood $V(d)\subset X$ and send it to a 2-sphere $\SS^2\in W_0$ representing element $a$ in $U\cong \pi_2(W_0,w_0)$, so that the boundary of the neighbourhood maps to the base point $w_0$. Neighbourhood of a trivalent vertex for lines $\sigma,\tau$ merging into $\sigma\tau$ is sent to the 2-simplex for $(\sigma,\tau)$ in the classifying space $BG\subset Y$. Furthermore, points of $X$ away from a tubular neighbourhood of lines and dots in $D$ are sent to the basepoint $w_0$. 
  
With this convention, relations on these networks in Section~\ref{subsec_monoidal_U} can be converted into rel boundary homotopies of maps $X\lra Y$ associated to networks. 
A 2-cocycle $c:G\times G\lra U$ can be suitably interpreted in this setup as well. 
 We plan to provide a detailed correspondence in another paper.  
\end{remark}

%
%

\section{Diagrammatics for two types of lines}
\label{section_two_lines}


\subsection{Crossed products of groups and overlapping networks of two kinds}\label{subsec_cross_prod}

$\quad$

Suppose $U$ is a group, not necessarily abelian, and we are given a group homomorphism $\psi:G\lra \Aut(U)$. This data $(U,G,\psi)$ determines a split extension of groups 
\begin{equation}\label{eq_UTG}
\xymatrix{
1 \ar[r] & U \ar[r] & T \ar[r] & G \ar@{..>}@/_1.25pc/[l]_s \ar[r] & 1, 
}
\end{equation}
or, equivalently, describes the semidirect product $T=U\rtimes G$. Its elements are $a\cdot\sigma$, with the multiplication $(a\cdot\sigma)(b\cdot\tau)=a\sigma(b)\cdot\sigma\tau$. A refinement of diagrammatics from Section~\ref{subsec_monoidal} gives a graphical calculus for a semidirect product, as follows. 

Although $G$ acts of $U$, since $U$ is no longer abelian one cannot use dots in the plane to denote its elements. Instead, we use co-oriented lines labelled by elements of $U$ which can merge and split, as in Section~\ref{subsec_monoidal}. These lines are shown as thin lines in black throughout this section, referred to as \emph{$U$-lines}, see Figure~\ref{fig_K1} top left.  

Elements of $G$ are labels of co-oriented lines of the second kind, depicted as thick wavy red lines. They can merge and split as well, in the usual way consistent with the setup of Section~\ref{subsec_monoidal}, see Figure~\ref{fig_K1} top middle. They are referred to as \emph{$G$-lines}. 
A $G$-line and a  $U$-line can intersect, and the label of the $U$-line is changed by the action of $G$ on it, see Figure~\ref{fig_K1} top right.  

\input{fig_K1}

There are a number of moves on these diagrams -- we leave working out a full set of them to the reader. Moves on networks 
with edges labelled by elements of a group were discussed in Section~\ref{subsec_monoidal}, and the moves there apply to both $U$-line and $G$-line networks. In general, a $U$-line and a $G$-line networks overlap, forming a single network. 
     
There are local moves for these overlapping networks. 
In particular, 
$G$-lines and $U$-lines and their vertices can slide through each other as shown in Figure~\ref{fig_K1}. 

\begin{remark}\label{remark_add}
Additionally, one can allow flip points and vertices of type II for both $G$-lines and $U$-lines, with the obvious modification rules, as discussed in Section~\ref{subsec_monoidal}.
\end{remark}

An alternative way to get a set of local moves is to  consider the simplest evaluation as in Section~\ref{subsec_monoidal}, where each closed diagram of overlapping $U$- and $G$-networks is evaluated to $0\in \{0\}$. For this evaluation any modification of the diagram that preserves its boundary is an allowed move. The result is the monoidal category $\mcC_T$ for the semi-direct product group $T$ in \eqref{eq_UTG}, but with a slightly different diagrammatic description, using the above two types of lines. 
 
\input{fig_K2}

A general element $a\cdot\sigma$ of $T$, $a\in U,\sigma\in G$ 
gives an object of $\mcC_T$ represented by a pair of dots on a horizontal line, labelled by $a$ and $\sigma$, respectively, as shown on the boundaries of Figure~\ref{fig_K2} on the left. 
The figure, also on the left, shows the identity morphism of this object, described by a  pair of parallel lines labelled $a$ and $\sigma$. 

An arbitrary object of $\mcC_T$ is represented by a finite sequence of elements of $U$ and $G$, together with co-orientation data at these points. Two sequences represent the same element of $T$ if and only if there is a morphism (network) between them, and there is at most one morphism between any two objects of $\mcC_T$. 

With these conventions, Figure~\ref{fig_K2} on the right shows 
a morphism from the object $(a,\sigma,b,\tau)$, with all co-orientations to the left, to the object $(a\sigma(b),\sigma\tau)$, again with left co-orientations. This morphism is given by overlapping $U$- and $G$-networks, each a merge. Note how the label on the $b$-edge changes to $\sigma(b)$ as the $U$-edge crosses $G$-edge $\sigma$. 

\vspace{0.07in}

Given a left $\Z[T]$-module $M$, the diagrammatics can be further extended by allowing elements of $M$ to float in regions, subjects to the rules described in Section~\ref{subsec_monoidal_U}. This gives rise to a category equivalent to $\mcC_{T,M}$ but with a slightly different diagrammatics. Constructions of Sections~\ref{subsec_one_cocycle} and~\ref{subsec_two_cocycle} work in this case as well. 

\vspace{0.07in}

Assume now that $U$ is abelian. One can then introduce a crossing of two $U$-lines, with either co-orientations, as the composition of a merge followed by a split, see Figure~\ref{fig_F5}. The crossings satisfy the usual permutation relations in the symmetric group, for either co-orientations.  

Furthermore, for abelian $U$, we can introduce dots labelled by $U$ floating in regions of the plane separated by the $G$-network. These dots can freely cross $U$-lines and they interact with $G$-lines as described in Section~\ref{subsec_monoidal_U}. Assume that the lines shows in figures in that section are red. Any point has a $G$-index as in \eqref{eq_winding} and a closed diagram evaluates to an element of $U$. Note that the $U$-network does not contribute to the evaluation, the evaluation is the sum over floating dots of their labels twisted by the action of $G$ via $G$-winding numbers of the labels. 

\vspace{0.07in} 

On the other hand, if $U$ is abelian, we may drop using $U$-networks and just revert to $U$-labelled dots and $G$-networks, as in Section~\ref{subsec_monoidal_U}. Multiplication in the central extension $T$ is then encoded as in Figure~\ref{fig_F23} but with the trivial cocycle $c(\sigma,\tau)=0\in U$ for all $\sigma, \tau\in G$.  The cocycle can be chosen to be trivial since the extension is split. We then recover category $\mcC_c$ as in Section~\ref{subsec_two_cocycle}, for the zero cocycle $c$.

\begin{remark}
One can additionally fix a one-cocycle $f\in \mathsf{Z}^1(G,U)$ (still assuming abelian $U$) and change the evaluation by adding contributions as in Figures~\ref{fig6_001} and~\ref{fig_F13} from the $G$-lines (which are wavy lines shown in red in this section). One then gets a different monoidal category, which we denote $\mcC_{G,U}^f$. It has more interactions between dots and $G$-lines. For example, a $G$-circle can be converted to a dot, as in Figures~\ref{fig_F10} and~\ref{fig_F11}. 
\end{remark} 


\subsection{Two examples: crossed modules and nonabelian one-cocycles}\label{subsec_cross_mod}

$\quad$ 

This section suggests  three examples of diagrammatical calculi, one for $G$-crossed modules and two for noncommutative one-cocycles. In these examples two (three in the last case) types of lines naturally appear, together with some interactions between them. We plan to explore these diagrammatics in a follow-up paper, and this section is not needed for the later parts of the paper. 

\vspace{0.1in} 

{\bf Crossed modules.}
A $G$-crossed module $(U,G,\partial)$ consists of a group homomorphism $\partial: U\lra G$ and an action of $G$ on $U$ subject to conditions: 
\begin{enumerate}
    \item $\partial$ is a homomorphism of $G$-groups: $\partial({}^\sigma a)=\sigma\partial(a) \sigma^{-1}$ for all $a\in U,\sigma\in G$, 
    \item the Pfeiffer identity holds: ${}^{\partial b}a = b a b^{-1}$ for all $a,b\in U$. 
\end{enumerate}
Crossed modules were introduced by Whitehead~\cite{Whitehead49}. 
For a subspace $Y\subset X$ the boundary homomorphism $\partial:\pi_2(X,Y)\lra \pi_1(Y)$ is a crossed module. In any crossed module (an easy exercise or see~\cite[Lemma~3.2.3]{Borovoi17}, for instance),
\begin{itemize}
    \item $\ker \, \partial$ is central in $U$, 
    \item $\mathrm{im}\, \partial$ is normal in $G$, 
    \item $G/\mathrm{im}\, \partial$ acts on $\ker\,\partial$.  
\end{itemize}

For relations of crossed modules to nonabelian cohomology see~\cite{Borovoi17} and references therein.  

Crossed modules admit a diagrammatics of overlapping networks that refines one for crossed products. Consider networks for cross-product of $G$ by $U$ coming from the action of $G$ on $U$, so that one has 
\begin{itemize}
\item 
Co-oriented thin lines labelled $a\in U$, with merges and splits as in Figure~\ref{fig_K1} top left. 
\item Co-oriented wavy thick lines (shown in red) labelled $\sigma \in G$, with merges and splits as well. 
\end{itemize}
At the intersection of a thin $a$-line and a thick $\sigma$-line the label $a$ of the thin line as it moves across the thick line becomes $\sigma(a)$, see Figure~\ref{fig_K1} top right. Other relations from  Figure~\ref{fig_K1} hold as well.
 We allow co-orientation reversal points at both thin and thick wavy lines as usual, see Figure~\ref{fig3_015}. 

Now add a vertex where a thin line becomes a thick wavy red line, with label changing from $a\in U$ to $\sigma(a)\in G$, co-orientations pointing in the same direction, see Figure~\ref{fig_K5}.

\vspace{0.07in}

\input{fig_K5}

\vspace{0.07in}

Impose the relations in Figure~\ref{fig_K5}. Note that lines are co-oriented, and a thin line $b$ in the lower center diagram turns into wavy line $\partial b$, then turns back into $b$.   One should be able to build a monoidal category associated with a crossed module from these diagrammatics.

\vspace{0.1in}

{\bf Nonabelian one-cocycles.}
Consider groups $G,U$ which are noncommutative, in general. Suppose given a homomorphism $\psi:G\lra \Aut(U)$. A one-cocycle of $G$ with values in $U$ is a map $f:G\lra U$ such that 
\begin{equation}\label{eq_noncomm}
     f(\sigma\tau)=f(\sigma)\cdot \sigma(f(\tau)), \ \ \ \sigma,\tau\in G. 
\end{equation}
We write $f\in \mathsf{Z}^1(G,U)$. One-cocycles $f_1,f_2$ are called \emph{equivalent} if for some $u\in U$, 
\begin{equation}
   f_2(\sigma) = c^{-1} f_1(\sigma) \sigma(c), \ \ \sigma\in G. 
\end{equation} 
The set of equivalence classes for this relation is denoted $\mathsf{H}^1(G,U)=\mathsf{Z}^1(G,U)/\sim$ and called the first nonabelian cohomology of $G$ with coefficients in $U$, see~\cite[Section 27]{Milne17} and~\cite{Serre02}. It is a group only if $U$ is commutative and otherwise has a distinguished element, the class of 1-cocycles of the form 
\begin{equation}\label{eq_one_twist}
  \sigma \ \longmapsto \ b^{-1}\cdot \sigma(b), \ \ b\in U,  
\end{equation} 
also called \emph{principal 1-cocycles}. 

\vspace{0.07in} 

To give a diagrammatic interpretation of a nonabelian 1-cocycle $f$, start with the diagrammatics for the groups $G,U$ and a homomorphism $G\lra \Aut(U)$ as in Section~\ref{subsec_cross_prod}. Encode the map $f:G\lra U$ by allowing a $G$-line labelled $\sigma$ to turn into a $U$-line labelled $f(\sigma)$ via a valency two vertex , together with reversing co-orientation at it, see Figure~\ref{fig_G1} on the left. 

\input{fig_G1}

The one-cocycle condition can now be interpreted as equality of diagram on the right of that figure. Note that the $f(\tau)$-line intersects the $\sigma$-line and gets rescaled to $\sigma(f(\tau))$. The label at the top of the middle diagram is $f(\sigma)\sigma(f(\tau))$, and the equality forces it to equal $f(\sigma\tau)$. 

\input{fig_G2}

Figure~\ref{fig_G2} shows how to interpret formula \eqref{eq_one_twist}, that is, a twist of a one-cocycle $f_1$ by a coboundary. One draws a line labelled $c$ and intersecting the $\sigma$-line. That changes $c$ at the intersection to $\sigma(c)$. The three smooth lines are then merged into one, which is $c^{-1}f_1(\sigma)\sigma(c)$, and the order of merges is irrelevant. Principal one-cocycles correspond to erasing the line $f_1(\sigma)$, with the diagram being a loop around the vertex where the $\sigma$-line terminates, intersecting $\sigma$-line once.  

Noncommutative one-cocycles are common in mathematics and appear, for instance, in classification of twisted algebraic structures  under Galois extensions~\cite{Serre02,Milne17}, see also~\cite{Borovoi17}. Classification of principal $G$-bundles is via Cech 1-cocycles and cohomology, which is different from cocycles in $\mathsf{Z}^1(U,G)$ that we consider, however. 

 \vspace{0.1in} 
 
The notion of a noncommutative 1-cocycle in ergodic theory and dynamical systems~\cite{Parry74,Katok79,Lema97,Arn98} requires a more specialized setup. Following~\cite{Katok79}, we consider a Lebesgue measure space $(X,\mu)$ and two locally compact second countable groups $\Gamma,G$. 
Assume that $\Gamma$ acts on $X$ in a measure-preserving way, denoting the action by $S_{\gamma}x$, $\gamma\in \Gamma$, $x\in X$. 

A $G$-cocycle over $S$ is a measurable function $\alpha:X\times \Gamma\lra G$ such that 
\begin{equation}\label{eq_coc_ergodic}
    \alpha(x,\gamma_1\gamma_2) \ = \ \alpha(x,\gamma_1)\alpha(S_{\gamma_1}x,\gamma_2), 
\ \ x\in X, \gamma_1,\gamma_2\in \Gamma. 
\end{equation}

\vspace{0.07in}

\input{fig_K6}

\vspace{0.07in}

Such cocycles can be interpreted diagrammatically, via planar networks with lines of three types: 
\begin{itemize}
\item A line labelled $x\in X$. Network of such lines can have trivalent vertices where three $x$-lines come together, see Figure~\ref{fig_K6}. 
\item Co-oriented lines labelled by elements $\Gamma$ that can form a trivalent network as usual, see Section~\ref{sec_planar}, shown as double lines in  Figure~\ref{fig_K6}. 
\item Co-oriented lines labelled by elements of $G$ and forming trivalent $G$-networks, as usual. 
\end{itemize}
An $X$-line can intersect a $\Gamma$-line, changing its label from $x$ to $S_{\gamma}x$, see Figure~\ref{fig_K6}. A trivalent vertex where all three types of lines come together denotes the 1-cocycle $\alpha$. The cocycle equation \eqref{eq_coc_ergodic} has a diagrammatic interpretation shown in Figure~\ref{fig_K6}. 

\vspace{0.07in}

\input{fig_K7}

Figure~\ref{fig_K7} shows diagrammatic for replacing a cocycle by a cohomologous cocycle. 

\vspace{0.07in}

\begin{remark}
Rewriting the map $\alpha:X\times\Gamma\lra G$ as a homomorphism $\alpha':\Gamma\lra \mathsf{Maps}(X,G)$ from $\Gamma$ to the group of measurable maps $X\lra G$ presents these 1-cocycles as a special case of 1-cocycles in \eqref{eq_noncomm} with $G$ there replaced by $\Gamma$ and $U$ by $\mathsf{Maps}(X,G)$.
\end{remark} 

Abelian and non-abelian 1-cocycles have recently emerged in statistics and graphical models~\cite{DB24} and in graph algorithms and neural networks~\cite{DGPV23}.  

\begin{remark}
    Fox derivatives $\partial_i:\Z F\lra \Z F$, where $F$ is a free group, with $\partial_i(x_j)=\delta_{i,j}$ and $\partial_i(fg)=\partial_i(f)+f\partial_i(g)$ are abelian 1-cocycles and play a fundamental role in knot and 3-manifold theory~\cite{RRV99,CF77}. 
\end{remark}

\begin{remark}
    Differential of a smooth map between manifolds can be interpreted as a non-commutative 1-cocycle, see \cite{Uzman22}. This example is sufficiently general so that many functors from a topological to a linear category can be interpreted as noncommutative 1-cocycles. In the one-dimensional case, the chain rule for the derivative is the 1-cocycle condition~\cite{AKM10} in the monoid of smooth functions $\R\lra \R$. For more sophisticated examples of cocycles in differential geometry we refer to~\cite{Tab96} and references therein. Some examples of uses of low-degree cocycles in quantum theory can be found in~\cite{Mickelsson2010cda,Mick10}. 
\end{remark}

%
%

\section{Entropy and infinitesimal dilogarithms} 
\label{section:entropy}


\subsection{Entropy of a finite random variable and its cohomological interpretation}
\label{subsec_entropy}

$\quad$ 

{\bf The entropy equation.}
Shannon's entropy of a probability distribution $X$ on a finite space $X=\{x_1,\ldots, x_n\}$ that associates  probabilities $p_1,\ldots, p_n$, $\displaystyle{\sum_{i=1}^n} p_i=1$, $0\le p_i\le 1$ to points $x_1,\dots, x_n$ is given by 
\begin{equation}\label{eq_shannon} 
H(X) \ := \ -\sum_{i=1}^n p_i 
\log p_i. 
\end{equation}

When $n=2$, one can think of $H(X)$ as a function of a single probability $p=p_1$, and  the entropy is  
\begin{equation}\label{eq_H_2}
H(p) \ = \ - p\log p - (1-p)\log (1-p).  
\end{equation}
It is natural to extend function $H$ from $0<p<1$ to all real numbers by  
\begin{equation}\label{eq_H_3}
H(p) \ := \ \begin{cases}
    
- p\log|p| - (1-p)\log |1-p| & \mathrm{if} \ p\not=0,1, \\
 0 & \mathrm{if} \ p=0,1. 
\end{cases}
\end{equation} 
Following Leinster~\cite{Lein21}, consider a related function $\psi:\R\lra \R$, denoted $\partial_{\R}$ in~\cite{Lein21} and given by 
\begin{equation}\label{eq_psi_func}
    \psi(a) \ = \ \begin{cases} - a\log |a| & \mathrm{if} \ a\not= 0, \\
    0 & \mathrm{if} \ a=0.
    \end{cases} 
\end{equation}
Function $\psi$ is a \emph{non-linear derivation}: 
\begin{equation}\label{eq_psi_derivation}
    \psi(ab) = \psi(a)b + a\psi(b), \ \ a,b\in \R, \ \ \psi(1)=0. 
\end{equation}
For a finite sequence $\underline{a}:=(a_1,\dots, a_n)$ of real numbers define 
\begin{equation} \label{eq_H_undx}
    H(\underline{a}):=\sum_{i=1}^n \psi(a_i) - \psi \biggl( \sum_{i=1}^n a_i\biggr) 
\end{equation}
to measure non-additivity of $\psi$ on the sequence $\underline{a}$. 
When $\underline{a}$ is a sequence of probabilities for a finite random variable $X$, the last term is $\psi(1)=0$, and one recovers the entropy of $X$. 

Entropy function $H:\R \lra \R$ satisfies a four-term functional equation
\begin{equation}\label{eq_H} 
H(p) - H(q) +p H\left(\dfrac{q}{p}\right) + (1-p)H\left(\dfrac{1-q}{1-p}\right) = 0, \ \ q\in \R, \ p\in \R,\ \ p\not=0,1  \end{equation} 
and 
\begin{equation}\label{eq_H_4}
    H(1-p)=H(p), \ \ p\in \R.  
\end{equation} 
The functional equation for entropy first appeared in~\cite{Tver58}. Introduction  to~\cite{Vign20} contains more information on the history of the functional equation. 
Equations \eqref{eq_H} and \eqref{eq_H_4}, normalization $H(1/2) = \log 2$,  and the continuity (or measurability) property uniquely determine $H$ as a  function $\R\lra \R$, see~\cite{Lee64}. 

 \vspace{0.1in}

 Setting $q=p$ in \eqref{eq_H} gives $H(1)=0$ and, by symmetry, $H(0)=0$. 
Setting $q=1$ in \eqref{eq_H} implies 
\begin{equation}\label{eq_H_5}
    p H\biggl(\frac{1}{p}\biggr) = - H(p), \ \ p\not=0 ,   
\end{equation}
and allows to rewrite \eqref{eq_H} in a more symmetric form 
\begin{equation}\label{eq_H_symm}
    H(p)+(1-p)H\biggl( \frac{q}{1-p}\biggr) \ = \ H(q)+(1-q)H\biggl( \frac{p}{1-q}\biggr), \ \ p,q\not= 1. 
\end{equation}

\begin{remark}
    Uniqueness of the entropy function is lost if measurability condition is dropped. Let $d:\R\lra \R$ be a derivation of $\R$ viewed as a $\Z$-algebra (so that $d(x+y)=dx+dy$ and $d(xy)=d(x)y+yd(x)$ for all $x,y\in \R$). Such derivations exist (via the Zorn Lemma) due to the transcendental structure of the field $\R$, see~\cite[Chapter II]{ZS58} or~\cite[Theorem 14.2.2]{Kuc09}. 
    
    Following~\cite{Aczel78}, consider the function 
\begin{equation}
\label{eq_bad_entropy}
h(x) \ = 
\begin{cases} 
\dfrac{(dx)^2}{x(1-x)} & \mathrm{if} \ x\not= 0,1, \\
0 & \mathrm{if} \ x=0,1. 
\end{cases}
\end{equation}
This non-measurable function satisfies functional equations \eqref{eq_H} and \eqref{eq_H_4}.
\end{remark}

{\bf Cathelineau--Kontsevich symbol.}
To relate entropy to group cohomology, following Kontsevich~\cite{Kont02} and Cathelineau~\cite{Cath88,Cath96,Cath07,Cath11}, consider a nonlinear symbol map $\langle \:\:, \:\: \rangle_H: \R \times \R \lra \R$ given by 
\begin{equation}\label{eq_symbol}
    \langle a,b\rangle_H \ := \ \begin{cases} (a+b)H\left(\dfrac{a}{a+b}\right)  & \mathrm{if} \ a+b\not= 0, \\
    0 & \mathrm{if} \  a+b=0. 
    \end{cases} 
\end{equation}
We have 
\begin{equation}\label{eq_angle_d}
    \langle a,b\rangle_H \ =\ - a \log|a| - b \log|b| + (a+b)\log|a+b|\ =\ \psi(a)+\psi(b)-\psi(a+b),  
\end{equation}
so that $\langle a,b\rangle_H$ measures the failure of $\psi$ to be additive for the pair $(a,b)$. Note that formula 
\begin{equation}\label{eq_angle_two}
    \langle a,b\rangle_H \ =\ \psi(a)+\psi(b)-\psi(a+b), \ \ \ \ \psi(a) \ = \ \begin{cases} - a\log |a| & \mathrm{if} \ a\not= 0, \\
    0 & \mathrm{if} \ a=0. 
    \end{cases}
\end{equation}
holds for all $a,b\in \R$.
A direct computation establishes the following result. 
\begin{prop}\label{prop_entropy_cocycle} The symbol satisfies the following relations 
\begin{eqnarray}
    \label{eq_ent_one}
    \langle a,b\rangle_H + \langle a+b,c\rangle_H & = & \langle a,b+c\rangle_H + \langle b,c\rangle_H, \\
    \label{eq_ent_two}
    \langle a,b\rangle_H & =& \langle b,a\rangle_H, \\
    \label{eq_ent_three}
    \langle ca,cb\rangle_H & = & c \langle a,b\rangle_H .
\end{eqnarray}
\end{prop}
The first relation above is the two-cocycle relation. The second says that the symbol is symmetric in $a$ and $b$, and the third is a scaling property. 

Writing down relation \eqref{eq_ent_one} for $a+b+c=1$ and replacing $p=a+b,q=a$ gives the equation
\begin{equation}
\label{eq_H_modified} 
H(p) - H(q) +p H\left(\dfrac{q}{p}\right) - (1-q)H\left(\dfrac{1-p}{1-q}\right) = 0, 
\ \ q\in \R, \ p\in \R, \  p\not=0,1  \end{equation}
This equation is similar to \eqref{eq_H} but the last term is different. Applying relation \eqref{eq_H_5} with $\frac{1-q}{1-p}$ as the argument converts \eqref{eq_H_modified} to  \eqref{eq_H}. 
 
To understand relation \eqref{eq_H_5} cohomologically, write \eqref{eq_ent_one} for $(a,b,c)=(p,1-p,p-1)$ to get 
\begin{equation}
    \langle p, 1-p\rangle_H + \langle 1, p-1\rangle_H = \langle p,0\rangle_H + \langle 1-p,p-1\rangle_H = 0+0 = 0. 
\end{equation}
Note that $\langle 1,p-1\rangle_H=p\, H(1/p)$ and $\langle p, 1-p\rangle_H=H(p)$, resulting in \eqref{eq_H_5}. 

\vspace{0.07in}

Equation \eqref{eq_ent_three} is a scaling property of the 2-cocycle. It indicates that one should add a scaling group $\R^{\ast}$ acting on $\R$. Form the cross-product group of affine symmetries of the real line 
\begin{equation}\label{eq_aff}
\Aff_1(\R) = 
\left\{ 
\begin{pmatrix}
c & a \\ 0 & 1
\end{pmatrix} : c\in \R^{\times}, a\in \R 
\right\},
\end{equation}  
with the multiplication in $\Aff_1(\kk)$ given by: 
\begin{equation}\label{eq_aff_mult}
\begin{pmatrix}
c_1 & a_1 \\ 0 & 1
\end{pmatrix}
\begin{pmatrix}
c_2 & a_2 \\ 0 & 1
\end{pmatrix} = 
\begin{pmatrix}
c_1c_2 & c_1 a_2 + a_1 \\ 0 & 1
\end{pmatrix}. 
\end{equation} 
To extend the 2-cocycle $\brak{a,b}$ to $\Aff_1(\R)$ we use the following formula 
\begin{equation}\label{eq_two_coc_o}
    \langle (a_1,c_1),(a_2,c_2)\rangle_H \ := \ \langle a_1, c_1 a_2\rangle_H,
\end{equation}
which will be motivated later, in Section~\ref{sec_entropy_diagramm}, from the  diagrammatic description. 
One can check that \eqref{eq_two_coc_o} is a (normalized)  2-cocycle on $\Aff_1(\R)$ 
valued in $\R$,  
\[
\Aff_1(\R)\times \Aff_1(\R)\lra \R,
\]
where the action of $\Aff_1(\R)$ on $\R$ is $(a,c)x= cx$. 

\vspace{0.07in} 

Relation~\eqref{eq_angle_two} shows~\cite{Kont02} that the 2-cocycle $\langle a,b\rangle_H$ on $\R$ is the coboundary of the 1-cochain $\psi:\R\lra \R$.

\begin{prop}  Two-cocycle
\eqref{eq_two_coc} on $\Aff_1(\R)$ is not a coboundary. 
\end{prop}
\begin{proof} This is due to lack of solutions $\phi:\R\times \R^{\ast}\lra \R$ of the equation 
\begin{equation}\label{eq_two_var}
\langle a_1,c_1a_2\rangle_H = \phi((a_1,c_1))+\phi((a_2,c_2)) - \phi((a_1+c_1a_2,c_1c_2)). 
\end{equation}
Indeed, any other solution $\psi_1$ of \eqref{eq_angle_two} has the form $\psi_1=\psi+\rho$ where $\rho:\R\lra \R$ is homomorphism (not necessarily continuous). Then $\phi((a,1))=\psi(a)+\rho(a)$ for some $\rho$ and all $a\in \R$. 
Specializing $c_1=1$ gives 
\begin{equation}\label{eq_two_var_one}
\langle a_1,a_2\rangle_H = \phi((a_1,1))+\phi((a_2,c_2)) - \phi((a_1+a_2,c_2)) 
\end{equation}
and 
\begin{equation}\label{eq_two_var_two}
\phi((a_2,c_2))-\phi((a_2,1)) =  \phi((a_1+a_2,c_2))-\phi((a_1+a_2,1))
\end{equation}
so that 
\begin{equation}\label{eq_psi_ac}
\phi((a,c)) = \phi((a,1))+\tau(c)=\psi(a)+\rho(a)+\tau(c)
\end{equation}
for some function $\tau:\R^{\ast}\lra \R$ and a homomorphism $\rho$. Substituting formula \eqref{eq_psi_ac} in \eqref{eq_two_var} quickly leads to a contradiction.  
\end{proof} 

Thus, the 2-cocycle \eqref{eq_two_coc} is a coboundary for the group $\R$ but not a coboundary for the larger group $\Aff_1(\R)$.


\subsection{Cathelineau's infinitesimal dilogarithm}
\label{subsection:Cathelineau_vect_space}

Infinitesimal dilogarithms were defined by Cathelineau~\cite{Cath88,Cath96} and Bloch--Esnault~\cite{BE03}. These two types of infinitesimal dilogarithms are slightly different. The relation is clarified in Bloch--Esnault~\cite{BE03} (also see  Garoufalidis~\cite{Garou09} for the case of a  characteristic $0$ field, a survey by  \"Unver~\cite{Unv21}, and Elbaz-Vincent and Gangl~\cite{EVG02}). Throughout the paper we consider the Cathelineau dilogarithm, following~\cite{Cath11,Cath96}. Bloch--Esnault dilogarithm carries an additional parameter of  higher degree which we do not discuss. 

Starting with a field $\kk$, 
Cathelineau defines $J(\kk)$ as a $\kk$-vector space generated by symbols $\brak{a,b}$, over $a,b\in\kk$, subject to the relations in Proposition~\ref{prop_entropy_cocycle}:
\begin{gather}
  \label{eq_brak_one}  \brak{a,b}  =  \brak{b,a}, \\
  \label{eq_brak_two}
    \brak{ca,cb}  =  c \brak{a,b}, \\
    \label{eq_brak_three}
    \brak{a,b} + \brak{a+b,c}  = \brak{b,c} + \brak{a,b+c} ,
\end{gather}
thus generalizing the symbol from $\R$ to an arbitrary field. 
The relations imply that $\brak{a,0}=\brak{0,a}=0$ and, if $\mathsf{char}(\kk)\not=2$, $\brak{a,-a}=0$. 
Cathelineau also introduces the $\kk$-vector space $\beta(\kk)$  generated by symbols $[a],$ $a\in \kk^{\times}$ subject to defining relations 
\[
[1]=0 
\]
and 
\begin{equation}\label{eq_4_term}
    [a]-[b]+ a\biggl[ \frac{b}{a}\biggr] + (1-a)\biggl[ \frac{1-b}{1-a} \biggr]=0, \ \ a\in \kk\setminus\{0,1\}, \ b\in \kk^{\times}.  
\end{equation}
We write $\kk^{\times}=\kk\setminus \{0\}$ for the multiplicative group of $\kk$ and $\kk^{\ast}$ for the $\kk$-vector space dual of $\kk$. 

Cathelineau proves that, if $\mathsf{char}(\kk) \not= 2$, there is an isomorphism of $\kk$-vector spaces 
\begin{equation}\label{eq_isoR}
R: \beta(\kk)\stackrel{\cong}{\lra} J(\kk), \ \ 
[a]\longmapsto \brak{a,1-a}
\end{equation} 
with the inverse map $R^{-1}$ given by 
\begin{equation}\label{eq_inverseR} 
\brak{a,-a}\longmapsto 0, \ \ \brak{a,b}\longmapsto (a+b)\biggl[ \frac{a}{a+b}\biggr] \ \ \mathrm{if} \ a+b\not= 0. 
\end{equation}

If $\mathsf{char}(\kk) = 2$, the isomorphism is given by 
\[
R_2: \beta(\kk)\lra J(\kk), \ \ [a]\longmapsto \brak{1,1+a}, 
\]
with the inverse 
\[
R_2^{-1}: J(\kk) \lra \beta(\kk), 
\]
We state Theorem 8.10 in \cite{Cath07}.

\begin{theorem}[Cathelineau]
Let $\Omega^1_\kk$ be the space of degree one absolute K\"ahler differentials. If $\mathsf{char}(\kk)=0$, the following sequence of $\kk$-vector spaces is exact. 
\begin{equation}\label{eq_ses_1}
    0 \lra \beta(\kk)\stackrel{f}{\lra} \kk^+\otimes \kk^{\times} \stackrel{g}{\lra} \Omega_{\kk}^1 \lra 0. 
\end{equation}
If $\mathsf{char}(\kk) = p>0$, the following sequence of $\kk$-vector spaces is exact
\begin{equation}\label{eq_ses_2}
    0 \lra \kk \stackrel{h}{\lra}\beta(\kk)\stackrel{f}{\lra} \kk^+\otimes \kk^{\times} \stackrel{g}{\lra} \Omega_{\kk}^1 \lra 0. 
\end{equation}
In these sequences the maps are 
\begin{eqnarray*}
& & f([a])=a\otimes a+(1-a)\otimes (1-a), \ \ \mathrm{for} \ a\not= 0,1; \ \ f([0])=0, \\
& & g(b\otimes a) = b \frac{da}{a}=b\, d \log(a), \ \ \ h(1)=\sum_{c\in \kk_0}[c].
\end{eqnarray*}
\end{theorem}
Here $\kk_0=\mathbf{F}_p$ is the prime subfield of $\kk$. 
The composition 
\begin{equation}
    f \circ R^{-1} \ : \ J(\kk)\lra \kk^+\otimes \kk^{\times}
\end{equation}
is given by 
\begin{equation}
    \brak{a,b} \longmapsto a\otimes a + b\otimes b-(a+b)\otimes (a+b)  . 
\end{equation}

\begin{remark}
Consider the cross-product group 
\begin{equation} \label{eq_affine_k}
\Aff_1(\kk) = 
\left\{ 
\begin{pmatrix}
c & a \\ 0 & 1
\end{pmatrix} : c\in \kk^{\times}, a\in \kk 
\right\}
\end{equation} 
of affine transformations of $\kk$ that send $x\mapsto cx+a$. Write its elements as $(a,c)$. This group fits into the short exact sequence of groups 
\begin{equation}
    0 \lra \kk \lra \Aff_1(\kk) \lra \kk^{\times}\lra 0 . 
\end{equation}
Denote by $\kk_r$ the \emph{left} $\Aff_1$-module which is $\kk$ with the action $(a,c).u=c^{-1}u$. Cathelineau~\cite[Theorem 6.1]{Cath11} establishes a canonical isomorphism 
\begin{equation}
    \Theta: \mathrm{H}_2(\Aff_1(\kk),\kk_r) \stackrel{\cong}{\lra} \beta(\kk),
\end{equation}
and tributes Chi-Han Sah with observing the isomorphism for infinite $\kk$. The isomorphism is induced by the map that takes a 2-chain $b[ (a,c)| (a',c')]$ in the non-homogeneous complex for the homology of $\Aff_1(\kk)$ to $b\langle c a',a\rangle \in J(\kk)$ and then to the corresponding element of $\beta(\kk)$ by composing with $R^{-1}$ in \eqref{eq_isoR}, see~\cite{Cath11} for details. 
\end{remark}

Summarizing, there are natural isomorphisms of $\kk$-vector spaces 
\begin{equation}\label{eq_iso_dilog}
    \beta(\kk) \ \cong \ J(\kk) \ \cong \ \mathrm{H}_2(\Aff_1(\kk),\kk_r) .
\end{equation}

Next, form $\kk$-vector space duals of the spaces and exact sequences \eqref{eq_ses_1} and \eqref{eq_ses_2}, in characteristic $0$ and $p$, respectively. The first sequence dualizes to the exact sequence (when $\mathsf{char}(\kk)=0$) 
\[
0 \lra \mathrm{Der}(\kk) \stackrel{g^{\ast}}{\lra} \mathrm{Log}(\kk)\stackrel{f^{\ast}}{\lra} \mathrm{Dil}(\kk)\lra 0, 
\]
and to the exact sequence 
\[
0 \lra \mathrm{Der}(\kk) \stackrel{g^{\ast}}{\lra} \mathrm{Log}(\kk)\stackrel{f^{\ast}}{\lra} \mathrm{Dil}(\kk)\stackrel{h^{\ast}}{\lra} \kk^{\ast}\lra 0 
\]
(when $\mathsf{char}(\kk)=p>0$). Here $\mathrm{Der}(\kk)$ is the vector space of derivations, that is, 
additive maps $d:\kk\lra \kk$ with $d(xy)=x \: d(y)+y \: d(x)$. The next term,  $\mathrm{Log}(\kk)$, the dual of $\kk\otimes \kk^{\times},$ can be identified with the space of group homomorphisms $\rho:\kk^{\times}\lra \kk^+$. 

The $\kk$-vector  space $\mathrm{Dil}(\kk):=J(\kk)^{\ast}$ is called the space of \emph{infinitesimal dilogarithms}.  

\vspace{0.1in}

Cathelineau~\cite{Cath11} establishes isomorphisms 
\begin{equation}\label{eq_iso_dilog2}
    \mathrm{Dil}(\kk) \ \cong \ \mathrm{Ext}^1_{\kk^{\times}}(\kk,\kk) \ \cong \ \mathrm{H}^2(\Aff_1(\kk),\kk_{\ell}).  
\end{equation}
Note that $\kk$ is naturally a $\Z[\kk^{\times}]$-module, and the corresponding ext group is considered in the isomorphism above. 
Left $\Aff_1(\kk)$-module 
$\kk_{\ell}\cong \kk$ has the action $(a,c) . u = cu$. 
 Isomorphism between the first and the last terms in \eqref{eq_iso_dilog2} is dual to the second isomorphism in \eqref{eq_iso_dilog}. 

\input{fig_entropy_001}

In 
Figure~\ref{fig_entropy_001} we write down various maps involved in the Cathelineau construction, for ours and the reader's convenience. 

\begin{remark} (Cathelineau~\cite{Cath11}) 
    If $\kk$ if a perfect field of finite characteristic $p$ then $\mathrm{Dil}(\kk)\cong \kk$ is one-dimensional, with a  generator 
    coming from the equivariant 2-cocycle $\chi^{\frac{1}{p}}$, where 
    \[
    \chi: \kk\times \kk \lra \kk, \ \ 
    \chi(x,y)=\sum_{i=1}^{p-1} \frac{1}{p}\binom{p}{i} x^i y^{p-i}. 
    \]     
    Cocycle $\chi$ also describes the addition on the space $W_2(\kk)$ of length 2 Witt vectors~\cite[Section 4.4]{Cath11}. The infinitesimal dilogarithm $\theta_p$ represented by the cocycle is given by 
    \[
    \theta_p(x)=\sum_{i=1}^{p-1} \frac{1}{p}\binom{p}{i} x^{i/p}(1-x)^{(p-i)/p}
    \]
    for $p\not=2$, $\theta_2(x)=1+x^{1/2}$, and it restricts to the Kontsevich $1\frac{1}{2}$ logarithm~\cite{Kont02,Lein21} on the prime subfield of $\kk$. 
    
     Cathelineau~\cite[Corollary 4.9]{Cath11} provides examples of fields where $J(\kk)$ is  infinite dimensional over $\kk$, including $\overline{\Q}$ and $F(X)$ for any field $F$.

  Cocycle $\chi$ above is closely related to the successful attack by Satoh--Araki, Smart, and Semaev~\cite{SA98,Smart99,Semaev98,Voloch} on the discrete logarithm problem on anomalous elliptic curves over $\mathbb{F}_p$ (i.e., curves with exactly $p$ points), see~\cite[Section 4]{GMP04}. 
\end{remark}

\begin{remark}
The entropy function $H$ is a rather special element of the $\R$-vector space 
\[\mathrm{Dil}(\R)=\Hom_{\R}(J(\R),\R) 
\]
given by 
\begin{equation}\label{eq_map_H}
H(\brak{a,b}) = \brak{a,b}_H. 
\end{equation}
Elements of  $\mathrm{Dil}(\R)$ not proportional to $H$ do not correspond to measurable functions. 
\end{remark} 

\begin{remark}
    A surprising connection between the Dehn invariant of polyhedra in $\R^3$ and the infinitesimal dilogarithm and entropy is developed in~\cite[Section 1.4]{Cath88}. The Dehn invariant $\mathcal{D}$ of a polyhedron in $\R^3$ is the first map in the scissor congruence short 
    exact sequence  
    \[
    0 \lra \mathcal{P}/\mathcal{L}\stackrel{\mathcal{D}}{\lra}\mathbb{U}\otimes_{\Z}i\R\lra \Omega_{\R}^1\lra 0 ,
    \]
    where $\mathcal{P}/\mathcal{L}$ is the $\R$-vector space of polyhedra in $\R^3$ modulo scissor congruence and prisms, see~\cite{Cath88,Syd65,Jess68,Schw13,Dup01,Sah79,Cale22_scissors}. 
    Cathelineau~\cite{Cath88} constructs a natural homomorphism from this short exact sequence into 
    the one in \eqref{eq_ses_1} for $\kk=\C$. Homomorphism of fields $\R\lra \C$ induces a 
    corresponding map of short exact sequences \eqref{eq_ses_1} from $\kk=\R$ to $\kk=\C$. This results in a diagram of three short exact sequences and maps between them, also see \cite{Cath88}. 
\[
\xymatrix@-1pc{
0\ar[rr] & &  \mathcal{P}/\mathcal{L} \ar[rr]^{\mathcal{D}} \ar[dd] & & 
i\R\otimes_{\Z}\mathbb{U}\ar[rr] \ar[dd] & & \Omega_{\R}^1\ar[rr] \ar[dd] & &  0 \\
& &  & &   & & \\ 
0 \ar[rr] & & \beta(\C) \ar[rr]^{f_{\C}} & & \C^+\otimes \C^{\times}\ar[rr]^{g_{\C}} & &  \Omega_{\C}^1 \ar[rr] & & 0  \\ 
& &  & &   & & \\ 
0 \ar[rr] & &  \beta(\R) \ar[rr]^{f_{\R}} \ar[uu] & &  \R^+\otimes \R^{\times} \ar[rr]^{g_{\R}} \ar[uu] & &   \Omega_{\R}^1 \ar[rr] \ar[uu] & &  0. 
}  
\]
    The entropy relates to the bottom left term in this diagram, being a function $H:\beta(\R)\lra \R$. 
    The analogue of symbols $\brak{a,b}$ in the scissor congruence group $\mathcal{P}/\mathcal{L}$ are special tetrahedra $T(a,b)$, for $0<a,b<1$. These are Euclidean tetrahedra $ABCD$ such that the edges $AB,BC,CD$ are pairwise orthogonal, $AB=\cot \alpha$, $BC=\cot \alpha\,\cot\beta$, $CD=\cot \beta$, where $0<\alpha,\beta<\frac{\pi}{2}$ are angles such that $\sin^2\alpha=a,\sin^2\beta=b$. One then has the 2-cocycle relation
    \[
    T(a,b)+T(ab,c) = T(b,c)+T(a,bc)
    \]
    in the scissor congruence group $\mathcal{P}/\mathcal{L}$, see~\cite{Jess68,Schw13}. The scaling law $T(ca,cb)=cT(a,b)$ and the commutativity condition $T(a,b)=T(b,a)$ hold as well (compare with the relations \eqref{eq_brak_one}-\eqref{eq_brak_three}), which allows to interpret the scissor congruence group cohomologically. We refer to~\cite[Section 1.4]{Cath88} for more details about this relation and to~\cite{Jess68,Schw13,Cale22_scissors} and the other references above for the background on the Dehn invariant. 

    In the above construction, parameters $a,b$ belong to a semigroup rather than a group. Extension of 2-cocycles from certain abelian semigroups to groups is studied in~\cite{JKT68, Ebanks79}. A similar extension of a 2-cocycle from $(\R_{>0},+)$ to $(\R,+)$ appears in \cite{IK24_SAF}, where the 2-cocycle is antisymmetric and comes from the cobordism group of unoriented weighted 1-foams embedded in $\R^2$. 

    D.~Calegari points out in~\cite{Cale22} that $-a\log a$ can be interpreted as the hyperbolic area of a suitable trapezoid in the upper half-plane, gives a scissor congruence proof of the entropy functional equation \eqref{eq_H}, and interprets his construction via a symplectic action of $\Aff^+_1(\R)\times \R$ on $T^{\ast}\R$.  

    We refer the reader to the end of Section 7 in~\cite{Sah79} for analogies pointed out to Sah by Chern between the Dehn invariant and structures in geometric probability, including integrals of mean curvature and quermassintegrale, in particular see formula (13.58) in~\cite{Santalo}. These formulas may also have been a motivation for Dehn to introduce his invariant. 
\end{remark}

%
%

\section{Diagrammatics of infinitesimal dilogarithms}
\label{sec_entropy_diagramm}

In this section we develop diagrammatics for the cohomological interpretation of infinitesimal dilogarithm, with an eye towards a similar interpretation of entropy in Section~\ref{subsec_diag_ent}. Following Cathelineau, we work over any field $\kk$ and use his universal 2-cocycle generated by symbols $\brak{a,b}\in J(\kk)$. Later on, we specialize the diagrammatics to $\kk=\R$ and the entropy 2-cocycle.  


\subsection{Motivation for the diagrammatics}
The 2-cocycle $\langle a,b\rangle$ with $a,b\in \kk$ can be represented diagrammatically as discussed in Section~\ref{subsec_two_cocycle}.
When lines $a,b\in \kk$ merge into the $(a+b)$-line, the vertex contributes $\langle a,b\rangle$ to the evaluation, see Figure~\ref{fig5_001} top left. 
It is natural to allow the opposite vertex, splitting $(a+b)$-line into $a,b$-lines. 
This vertex should then contribute $-\langle a,b\rangle$ to the evaluation, see Figure~\ref{fig5_001} top middle. 
As usual, the 2-cocycle property in \eqref{eq_ent_one} corresponds to relation in Figure~\ref{fig5_002} top left.  

\input{fig5_001}

\input{fig5_002}

\vspace{0.07in}

The symmetry property \eqref{eq_ent_two} indicates that we should allow the lines to intersect, with the intersection point contributing $0$ to the evaluation, but adding a skein relation in Figure~\ref{fig5_002} top right, i.e., $\langle b,a\rangle = \langle a, b\rangle$. The intersection can also be defined as a composition of a split and a merge, as shown in Figure~\ref{fig5_103} in top left. 
We refer to these lines as \emph{additive lines} and to split and merge vertices as \emph{additive vertices}, and in this section use oriented rather than co-oriented additive lines, for convenience. 

\vspace{0.07in}

In the previous section, we discussed the scaling property \eqref{eq_brak_two}  of the 2-cocycle $\brak{a,b}$ and passed to the affine group $\Aff_1(\kk)$ in \eqref{eq_affine_k}, which is a cross-product of $\kk^{\times}$ and $\kk$. Diagrammatics of cross-product groups was introduced in Section~\ref{subsec_cross_prod}. We follow that diagrammatics as a guide and introduce lines of the second kind  -- \emph{multiplicative} lines, labelled by elements of $\kk^{\times}$. These are shown as red wavy lines in the diagrams, labelled by elements $c\in \kk^{\times}$ and co-oriented. 
At an intersection point of an additive and a multiplicative line, the weight of the additive line is scaled by the weight of the multiplicative line, see the bottom left diagram in Figure~\ref{fig5_001}. The scaling happens in the direction of the co-orientation of the multiplicative line. 

We allow multiplicative lines to merge and split, while respecting their co-orientation and multiplying the labels when two lines merge or a line splits into two lines, see the bottom middle Figure~\ref{fig5_001}.

 Thus, we consider planar networks of diagrams where lines have two types: 
\begin{itemize}
    \item \emph{Additive lines} that are oriented and carry elements of $\kk$.  Two such lines labelled $a,b$ can merge into a line $a+b$, preserving orientation. Additive lines can intersect. They can also merge or split, preserving orientation, with their weights added.  
    \item \emph{Multiplicative lines}, shown as wavy red lines, that are co-oriented and carry an element of $\kk^{\times}$. Multiplicative lines can merge or split, respecting co-orientations, with the weights multiplied in a merge (and likewise for a split). 
    \item Arbitrary isotopies of diagrams are allowed.  
\end{itemize}

Co-orientation reversal nodes are allowed on multiplicative lines, see Figure~\ref{fig5_001} bottom right. At such a node label $c$ is reversed to $c^{-1}$. A merge of two multiplicative lines with the co-orientations opposite to that in Figure~\ref{fig5_001} is allowed as well. Splits of a red line into two are given by rotating the corresponding merge diagram in Figure~\ref{fig5_001} (middle of bottom row) by $180^{\circ}$.

\input{fig5_012}

\vspace{0.07in} 

Multiplicative (red) $1$-line can be erased, see Figure ~\ref{fig5_022}. Points where it splits into $c$ and $c^{-1}$ multiplicative lines become co-orientation reversal nodes, as in Figure~\ref{fig5_001} bottom right. 

\input{fig5_022}

\vspace{0.07in} 

Additive vertex in Figure~\ref{fig5_001}  top left that contributes $\langle a,b\rangle$ represents the addition operation in $\kk$. It extends to the group operation in the cross-product as shown in the multiplication diagram in   
Figure~\ref{fig5_012} for $\Aff_1(\kk)$ or in Figure~\ref{fig_K2} on the right (for general cross-products). In Figure~\ref{fig5_012} the bottom endpoints alternate between additive and multiplicative, and they are $a_1,c_1,a_2,c_2$ in that order. In the diagram additive $a_2$-line crosses multiplicative $c_1$-line and gets rescaled to additive $c_1a_2$-line. After that the two additive lines become adjoint and can be merged into the additive line labelled $a_1+c_1a_2$. Likewise, multiplicative lines become adjoint and can be moved into the multiplicative line $c_1c_2$. At the top boundary of the diagram we see an additive endpoint $a_1+c_1a_2$ followed by the multiplicative endpoint $c_1c_2$. Note that all three additive endpoints at the boundary of the diagram are upward oriented and all three multiplicative endpoints at the boundary are left-cooriented. The diagram matches the multiplication rule \eqref{eq_aff_mult} in the affine group $\Aff_1(\kk)$. 


\subsection{Monoidal categories for the infinitesimal dilogarithm}
\label{subsec_mon_for_dilog}

We would like to define two pivotal monoidal categories $\Catone$, $\Cattwo$ that are based on the above diagrammatical calculi. 
Generating objects in both categories correspond to possible endpoints in the diagrams. The objects are $X^+_a,X^-_a$, $a\in \kk$ and $Y^+_c,Y^-_c$, $c\in \kk^{\times}$. 
An object $Z$ of $\Catone$ or $\Cattwo$ is a finite tensor product of objects of these four types over various $a$'s and $c$'s. The empty product is the unit object $\one$. Generating 
 morphisms in $\Catone$ include the morphisms in  Figure~\ref{fig5_001} and pivotal structure morphisms (\emph{cups} and \emph{caps}) in Figure~\ref{fig4_020}. 

\input{fig4_020}

 In more detail, the generating morphisms in $\Catone$ are the following: 
\begin{itemize}
    \item Morphisms for the pivotal structure, with $X^-_a$ the dual of $X^+_a$ and $Y^-_c$ the dual of $Y^+_c$, see Figure~\ref{fig4_020}.
    \item Merge and split morphisms $\rho_m: X^+_a\otimes X^+_b\lra X^+_{a+b}$ and $\rho_s:X^+_{a+b} \lra X^+_a\otimes X^+_b$ 
    for additive lines shown in the top row (left and middle) of Figure~\ref{fig5_001}. 
    Dual merge and split morphisms $X^-_b\otimes X^-_a\lra X^-_{a+b}$ and $X^-_{a+b} \lra X^-_b\otimes X^-_a$ 
     can be obtained by rotating diagrams for $\rho_s$ and $\rho_m$ by $180^{\circ}$. 
    \item Permutation morphism $X^+_a\otimes X^+_b\lra X^+_b\otimes X^+_a$ in Figure~\ref{fig5_001} top right and its variations given by reversing one or both orientations, corresponding to intersections of additive lines. (We do not introduce intersections of multiplicative lines, although it is possible to do that due to commutativity of $\kk^{\times}$.) 
    \item Intersection morphism $Y^+_c\otimes X_a^+\lra X_{ac}^+\otimes Y^+_c$ shown in Figure~\ref{fig5_001} bottom right and its variations given by reversing the orientation of the additive line and reversing the co-orientation of the multiplicative line.
    \item Merge and split morphisms for the multiplicative lines co-oriented to the left 
    $Y^+_{c_1}\otimes Y^+_{c_2}\lra Y^+_{c_1c_2}$, $Y^+_{c_1c_2}\lra Y^+_{c_1}\otimes Y^+_{c_2}$ and co-oriented to the right 
    $Y^-_{c_1}\otimes Y^-_{c_2}\lra Y^-_{c_1c_2}$, $Y^-_{c_1c_2}\lra Y^-_{c_1}\otimes Y^-_{c_2}$. 
    \item Co-orientation reversal morphisms $Y^+_c\lra Y^-_{c^{-1}}$ and $Y^-_c\lra Y^+_{c^{-1}}$ in  Figure~\ref{fig5_001} bottom right. 
\end{itemize}

In the category $\Cattwo$, we also include generating morphisms $d(a)\in \End(\one)$, for $a\in \kk$, represented by a dot labelled $a$. 

Categories $\Catone$ and $\Cattwo$ are pivotal categories, with $X_a^-$ the two-sided dual of $X_a^+$ and $Y_c^-$ the two-sided dual of $Y_c^+$. The pivotal structure comes from allowing arbitrary isotopies of diagrams, which does not change the underlying morphism. 
We first define $\Catthree$, respectively, $\Catfour$,   to be the free pivotal monoidal category with the above generating morphisms, without, respectively with, morphisms $d(a)$, $a\in \kk$. These categories are an intermediate step in our construction. 
An example of a morphism in $\Catthree$ is shown in Figure~\ref{fig4_018}. A morphism in $\Catfour$ may additionally contain dots floating in the regions of the diagram and labelled by elements of $\kk$. 

\input{fig4_018}

\input{fig5_006}

\input{fig4_022}

We then define $\Catone$ and $\Cattwo$ as suitable monoidal quotients of $\Catthree$ and $\Catfour$, respectively. 
One possible definition is to pick a full set of defining relations on our networks and impose these relations. 
Some of the natural relations on the networks are shown in Figure~\ref{fig5_002}. One can add to these the permutation relations on additive crossings, relations that involve co-orientation reversal on multiplicative lines (the latteris shown in Figure~\ref{fig5_001} bottom right). Furthermore, it is natural to add relations for the cross-product diagrammatics shown in Figure~\ref{fig_K1} and specialized to $\Aff_1(\kk)$, so that $\sigma,\tau$ in the latter figure are elements of $\kk^{\times}$ scaling $a,b\in \kk$. For a complete match, one should convert co-orientations on additive lines in Figure~\ref{fig_K1} to orientations in a uniform way. 

Given a diagram for a morphism in $\mcC_{\kk}$, the multiplicative network in the diagram can be moved freely relative to the additive network, rel boundary (preserving all boundary points of the diagram), and taking into account changes in additive weights as additive lines cross multiplicative lines. Figures~\ref{fig5_006} and~\ref{fig4_022} show examples of such moves. 

In the category $\Cattwo$, where dots are present, we should add analogues of the relations in Figure~\ref{fig_F10a}: (1) dots labelled $x,y$ and floating in the same region can be merged into $x+y$ dots; (2) moving an $x$-dot through a red line $c$ co-oriented away from the dot scales the dot by $c$ to an $cx$-dot.

\vspace{0.07in} 

To avoid listing a full defining set of relations we instead define morphisms in these categories using a suitable evaluation function and the universal construction. 
Consider a planar diagram, representing a morphism in $\Catone$ from a tensor product $Z$ to $Z'$ and built out of the above generators, see Figure~\ref{fig4_019} for a representative example. 

\input{fig4_019}

In Section~\ref{subsec_entropy} we defined the $\kk$-vector space $J(\kk)$ associated to the field $\kk$. 
To a morphism $\gamma$ in $\Catone$ from $Z$ to $Z'$ assign an element $\jmath(\gamma)\in J(\kk)$ as follows. First, each point $p$ of the diagram of $\gamma$ not on a multiplicative line has the $\kk^{\times}$-winding number  $\omega(p,\gamma)\in \kk^{\times}$ relative to the leftmost region of $\gamma$. It is defined just as in Section~\ref{subsec_monoidal}, also see Figure~\ref{fig3_020a}. To compute it, draw a generic path from $p$ to the leftmost region of the diagram and multiply labels $c$ of wavy (red) lines that the path crossed, using the inverse $c^{-1}$ if the co-orientation of the wavy line at the intersection points towards $p$. 

The element $\jmath(\gamma)$ is a sum over all merges and splits of additive lines. Denote by $\mathrm{add}(\gamma)$ the set of such merges and splits. For $p\in \mathrm{add}(\gamma)$, define its sign $s(p)=1$ if $p$ is a merge and $s(p)=-1$ if $p$ is a split. Denote by $a_p$ and $b_p$ the labels $a$ and $b$ of the two lines that merge or split, following the order in Figure~\ref{fig5_001} top left and middle diagrams.

Define 
\begin{equation}\label{eq_func_j}
\jmath(\gamma) = \sum_{p\in \mathrm{add}(\gamma)} s(p) \omega(p,\gamma)\langle a_p,b_p\rangle \ \in \ J(\kk).
\end{equation}
As an example, consider the morphism $\gamma$ in Figure~\ref{fig4_019}, with three additive vertices, which we label $p_1,p_2,p_3$ from left to right. The leftmost point $p_1$ is a split, so $s(p_1)=-1$, and has $\kk^{\times}$-winding number $1$, since it is already in the leftmost region of the  wavy red network. It contributes $-\langle a_1,a_2\rangle$ to $\jmath(\gamma)$. Point $p_2$ in the middle is a merge, $s(p_2)=1$, the winding number $\omega(p_2,\gamma)=c_1$, and it contributes $c_1\langle c_1^{-1}a_2,a_3\rangle=\langle a_2,c_1a_3\rangle$ to $\jmath(\gamma)$. Finally, the rightmost point $p_3$ has $s(p_3)=-1$ and $\omega(p_3,\gamma)=c_1c_2$. The invariant $\jmath(\gamma)$ is written down in Figure~\ref{fig4_019}. 

Define $\Catone$ to be the quotient of $\Catthree$ by the relation: two morphisms $\gamma_1,\gamma_2: Z\lra Z'$ in $\Catthree$ are equal as morphisms in $\Catone$ if and only if $\jmath(\gamma_1)=\jmath(\gamma_2)$. To be precise, we first define the hom set $\Hom_{\Catone}(Z,Z'):= \Hom_{\Catthree}(Z,Z')/\sim$ as the quotient of the corresponding hom set in $\Catthree$ by this equivalence relation and later (Proposition~\ref{prop_composable}) check that composition in $\Catone$ is well-defined. 

\input{fig5_103}

\input{fig5_105}

\input{fig4_021}

\begin{prop} Relations in Figures~\ref{fig5_002},~\ref{fig5_103},~\ref{fig5_105}, and~\ref{fig4_021} hold in $\Catone$. Modifications of these relations given by reversing orientations of additive lines or the co-orientation of the multiplicative line (simultaneously with replacing $c$ by $c^{-1}$, bottom row in Figure~\ref{fig5_002}) hold as well. 
\end{prop}

\begin{proof} Top left relation is the 2-cocycle equation \eqref{eq_brak_three}. Orientation reversal corresponds to adding minus sign to all terms of the equation, thus the corresponding diagrammatic relation holds as well.  Top right relation is the commutativity \eqref{eq_brak_one}. Bottom left relation is the equation~\eqref{eq_brak_two}.  Bottom right relation holds since both diagrams contribute $0$ to $\jmath(\gamma)$. 
\end{proof}

\begin{prop}\label{prop_long_proof}
    Suppose that $\gamma:\one\lra \one$ is an endomorphism of the identity object in $\Catone$. Then $\jmath(\gamma)=0$. 
\end{prop}
\begin{proof} An endomorphism $\gamma$ is given by two overlapping networks of additive and multiplicative lines, with the empty boundary. Using relations discussed above we can isotop the additive network away from the multiplicative network (possibly changing weights on additive lines in the process). After the isotopy additive and multiplicative networks  become disjoint. The multiplicative network evaluates to $0$. An additive network $\gamma_0$ evaluates to the signed sum  $\jmath(\gamma_0)=\sum_{p\in \mathrm{add}(\gamma_0)} s(p) \langle a_p,b_p\rangle  \in  J(\kk),$ since $\kk^{\times}$-winding numbers of its vertices are all $1$. We claim that $\jmath(\gamma_0)=0$. This can shown by converting $\gamma_0$ to a network $\widehat{\gamma_1}$ which is the closure of a braid-like network $\gamma_1$, with all lines pointed up and no U-turns. An example of a braid-like additive network and its closure is shown in Figure~\ref{fig4_023}. In a braid-like (additive) network all lines are oriented upward so that no tangent line to the  network is horizontal, and the network is a composition of merges, splits and transpositions of lines.  

The Alexander theorem~\cite{Bir74} says that any oriented link in $\R^3$ is the closure of a braid. 
It is straightforward to derive a weak analogue of the Alexander theorem: for any closed additive network $\gamma_0$, there exists a braid-like network $\gamma_1$ such that $\jmath(\gamma_0)=\jmath(\widehat{\gamma_1})$. This can be proved by mimicking the proof of the Alexander theorem:
\begin{itemize}
    \item Convert $\gamma_0$ to a piecewise-linear (PL) form and choose a generic base point $p_0$ in the plane. Further deform $\gamma_0$ so that near each intersection, merge and split point orientations of intervals of $\gamma_0$ all point clockwise around $p_0$. 
    \item Any interval in $\gamma_0$ which goes clockwise around $p_0$ is left alone. An interval $I$ that goes counterclockwise around $p_0$ is modified into a PL curve by moving some portions of $I$ towards $p_0$ and some all the way around the diagram as schematically shown in Figure~\ref{fig4_024}. These moves do not change the invariant $j$.
\end{itemize} 

\input{fig4_024}
 
Network 
$\widehat{\gamma_1}$ is analogous to  a braid closure. We claim that  $\jmath(\widehat{\gamma_1})=0$ for any braid-like network $\gamma_1$. 
To see this, consider merges, splits, and transposition points of $\gamma_1$. Put them in general position so that their $y$-coordinates are all different. 
Using moves that do not change the $j$-invariant we can modify $\gamma_1$ to a braid-like network $\gamma_2$ with no transposition points by replacing the latter with merges followed by splits as in Figure~\ref{fig5_103} bottom left. Then, using isotopies and the move in Figure~\ref{fig5_105}, network $\gamma_2$ is modified to a braid-like network $\gamma_3$ which consists of merges and splits with all merges happening before the splits (all having smaller $y$-coordinates than those of the split points). The bottom and top boundaries of $\gamma_3$ are given by the same sequence $a_1,\ldots, a_k\in \kk$ of weights. 

Network $\gamma_3$ splits the $i$-line, of weight $a_i$, into a sequence of lines of weights $a_{i,1},\ldots, a_{i,\ell_i}$. Then these lines, over $i=1,\ldots, k$ and in this order are merged back into the same sequence $a_1,\ldots, a_k$. Using associativity moves (Figure~\ref{fig5_002} top left and its reflection) and cancellation moves in Figure~\ref{fig4_021} network $\gamma_3$ can be reduced to a network $\gamma_4$ which consists of vertical lines of weights $a_1, \ldots, a_k$, with two consecutive lines (of weights $a_i,a_{i+1}$) exchange at most one line of weight $0$, see Figure~\ref{fig4_025}. Then $\jmath(\widetilde{\gamma_4})=0$ since $\langle a,0\rangle=0$ for all $a\in \kk$. 

\input{fig4_025}

\input{fig4_023}

Another way to show that a braid-like network $\gamma$ has zero $\jmath$ invariant is to  merge all upward-oriented lines of $\widehat{\gamma}$ that intersect a given horizontal line into a single line and then split it back (this does not change the invariant).  
The network then starts and ends with a single line and it can be gradually simplified to the vertical line, without changing the invariant. A simple example is shown in  
Figure~\ref{fig5_104}.

One then gradually simplifies this modification of $\gamma_1$ without changing $\jmath$ until it is a vertical diagram, with $\jmath=0$. An example of such modification, reducing a  crossing squeezed between a full split and a full merge, is shown in Figure~\ref{fig5_104}. 
\input{fig5_104}
\end{proof}

\begin{remark} One should be able to prove the above proposition more directly, using a variation on Proposition~\ref{prop_closed_diagram}. Here we use two networks, one for the multiplicative and one for the additive subgroups of $\Aff_1(\kk)$, while Proposition~\ref{prop_closed_diagram} uses a single network for all elements of a group $G$. 
\end{remark}

\begin{corollary}
We have the following: 
\begin{enumerate}
\item  Given a morphism $\gamma:Z_0\lra Z_1$ in $\Catone$, the invariant $\jmath(\gamma)$ depends only on $Z_0$ and $Z_1$. 
\item $\jmath(\gamma)=0$ for any endomorphism $\gamma$ of an object $Z$.  
    \end{enumerate}
\end{corollary}

\begin{proof}
Consider two morphisms $\gamma,\gamma':Z_0\lra Z_1$. Reflect the morphism $\gamma$ about the $y$-axis and reverse the orientation of all additive lines. Denote the resulting morphism by $\gamma^{\ast}:Z_0^{\ast}\lra Z_1^{\ast}$ ($\ast$ operation on morphisms is part of the pivotal structure on $\Catone$). 
Note that $\jmath(\gamma^{\ast}) = -c(Z)^{-1}.\jmath(\gamma)$.

\input{fig4_026}

Compose morphisms $\gamma,\gamma'$ with $\gamma^{\ast}$ and the duality morphisms and  to get endomorphisms $\tau,\tau'$ of $\one$, see Figure~\ref{fig4_026}. We have 
\begin{eqnarray}
\jmath(\tau) & = & \jmath(\gamma)+ c(Z).\jmath(\gamma^{\ast})=
\jmath(\gamma)-c(Z).c(Z)^{-1}.\jmath(\gamma)=0,\\
\jmath(\tau') & = &
\jmath(\gamma')+ c(Z).\jmath(\gamma^{\ast})=
\jmath(\gamma')-\jmath(\gamma)=0.
\end{eqnarray}
The last equality holds since $\jmath$ is $0$ on any endomorphism of $\one$, including $\tau'$. Consequently, $\jmath(\gamma')=\jmath(\gamma)$ and part (1) of the corollary follows. 

Part (2) holds since $\jmath(\id_Z)=0$, where $\id_Z$ is the identity endomorphism of $Z$ (the diagram of $\id_Z$ consists of vertical parallel lines, which do not contribute to $\jmath$). Consequently, $\jmath(\gamma)=0$ for any endomorphism $\gamma$ of $Z$.
\end{proof}

An object $Z\in \Ob(\Catone)$ is a tensor product of objects $X^{\pm}_a$ and $Y^{\pm}_c$, over various $a\in \kk$ and $c\in \kk^{\times}.$ Define \emph{the weight} of $Z$, denoted $w(Z)$, to be the unique pair $(a(Z),c(Z))$, with $a(Z)\in \kk$ (\emph{additive weight}) and $c(Z)\in \kk^{\times}$ (\emph{multiplicative weight}) such that there exists a morphism 
\begin{equation}\label{eq_Z_morphism}
    Z\lra X_{a(Z)}^+\otimes Y_{c(Z)}^+
\end{equation}
in $\Catone$. 
To determine the weight of $Z$, first move all $X$-terms in the product for $Z$ to the left of all $Y$-terms using permutation morphisms 
\begin{equation}\label{eq_perm}
Y^{\pm}_c \otimes X^+_a \lra X^+_{c^{\pm 1}a}\otimes Y^{\pm}_c, \ \ 
Y^{\pm}_c \otimes X^-_a \lra X^+_{c^{\mp 1}a}\otimes Y^{\pm}_c 
\end{equation} 
given by the morphism in  Figure~\ref{fig5_001}, bottom left corner, and its modifications with reversed orientation and co-orientation. 
The new product 
\[
Z_0=X^{\epsilon_1}_{a_1}\otimes \cdots \otimes X^{\epsilon_n}_{a_n}\otimes Y^{\mu_1}_{c_1}\otimes \cdots \otimes  Y^{\mu_m}_{c_m}, \ \ 
\epsilon_i,\mu_j\in \{+,-\}. 
\]
Define 
\begin{equation}
a(Z)  = \sum_{i=1}^n \epsilon_i a_i\in \kk, \ \ c(Z) = \prod_{j=1}^m c_j^{\mu_j 1}\in \kk^{\times}, 
\end{equation} 
where in the formula $\mu 1$ stands for $1$ or $-1$ depending on whether $\mu=+$ or $\mu=-$, and likewise for $\epsilon a$. 
With this choice of parameters there is a morphism \eqref{eq_Z_morphism}. The weight $w(Z)=(a(Z),c(Z))$ of $Z$ is naturally an element of the affine group $\Aff_1(\kk)$. 

Weights of the tensor product $Z\otimes Z'$ are given by: 
\begin{equation}
a(Z\otimes Z')= a(Z) + c(Z) a(Z'), \ \ c(Z\otimes Z')=c(Z)c(Z'), 
\end{equation}
which we can also write as $w(Z\otimes Z')=w(Z)w(Z')$, where the product is just the multiplication in $\Aff_1(\kk)$. 

Weights of the source and target objects of any morphism $\gamma:Z\lra Z'$ in $\Catone$ are equal: $a(Z)=a(Z'),c(Z)=c(Z')$. 
 The following propositions are clear.

 \begin{prop}\label{prop_morphism_exists}
A morphism $\gamma:Z\lra Z'$ from $Z$ to $Z'$ in $\Catone$ exists if and only if the weights of $Z$ and $Z'$ are equal. 
\end{prop} 

\begin{prop} \label{prop_composable}
For composable morphisms, $\jmath(\gamma_2\gamma_1)=\jmath(\gamma_2)+\jmath(\gamma_1)$. For two morphisms $\gamma_i:Z_i\lra Z_i'$, $i=1,2$, the invariant 
$\jmath(\gamma_1\otimes \gamma_2)=\jmath(\gamma_1)+ c(Z_1)\jmath(\gamma_2).$
\end{prop} 

\begin{corollary} Composition of morphisms in $\Catone$ is well-defined. Category $\Catone$ is pivotal monoidal. 
\end{corollary}

\begin{proof}
    If $\jmath(\gamma_0)=\jmath(\gamma_0')$ for $\gamma_0,\gamma_0':Z_0\lra Z_1$ and $\jmath(\gamma_1)=\jmath(\gamma_1')$ for $\gamma_1,\gamma_1':Z_1\lra Z_2$, then $\jmath(\gamma_1\gamma_0)=\jmath(\gamma_1'\gamma_0')$ due to Proposition~\ref{prop_composable}. This implies that composition in $\Catone$ is well-defined and $\Catone$ is a category. 

    The monoidal structure in $\Catone$ is inherited from that of $\Catone$. That it is well-defined follows from the formula for the $\jmath$ invariant of the tensor product in Proposition~\ref{prop_composable}.
\end{proof}

\begin{corollary}\label{cor_unique}
    Any morphism $\gamma$ in $\Catone$ is invertible. For any two objects $Z,Z'$ of $\Catone$ there exist at most one morphism $\gamma: Z\lra Z'$ in $\Catone$. A morphism $\gamma:Z\lra Z'$ exists if and only if $Z$ and $Z'$ have equal weights: $w(Z)=w(Z')$. For any object $Z$ there exists an isomorphism 
    \begin{equation}
    Z\stackrel{\cong}{\lra} X^+_{a(Z)}\otimes Y^+_{c(Z)}.
    \end{equation}
\end{corollary}

\begin{proof} Two morphisms $\gamma_0,\gamma_1: Z\lra Z'$ have the same $\jmath$ invariant: $\jmath(\gamma_0)=\jmath(\gamma_1)$ so they are equal. The inverse of $\gamma$ is given by reflecting $\gamma$ in a horizontal plane. The existence condition follows from that in Proposition~\ref{prop_morphism_exists} for the category $\Catone'$.  
\end{proof} 

Category $\mcC_{\kk}$ is small and morphisms in it are in a bijection with pairs of objects $(Z,Z')$ of equal weight. The $\jmath$ invariant of morphisms in $\mcC_{\kk}$ takes values in the $\kk$-vector space $J(\kk)$. It extends the entropy and the Cathelineau invariant to morphisms in  $\mcC_{\kk}$. 

To compute $\jmath(\gamma)$ for a morphism $\gamma:Z_0\lra Z_1$ (equivalently, for a pair $(Z_0,Z_1)$ with $a(Z_0)=a(Z_1)$ and $c(Z_0)=c(Z_1)$), factor $\gamma$ as 
\begin{equation}\label{eq_factor_Z}
Z_0 \stackrel{\gamma_0}{\lra} X^+_{a(Z_0)}\otimes Y^+_{c(Z_0)}\stackrel{\gamma_1}{\lra}  Z_1
\end{equation}
Morphism $\gamma_0$ can be represented by the diagram where all additive points of $Z_0$ cross to the left of all multiplicative points of $Z_0$. Downward-oriented additive lines then reverse orientation (and their weights). Right-co-oriented multiplicative lines reverse coorientation and their weights. Finally all additive lines merge into a single upward additive line that carries weight $a(Z_0)$ and all multiplicative lines merge into a single left cooriented line that carries weight $c(Z_0)$, see Figure~\ref{fig8_010}. Additive weights get rescaled upon crossing multiplicative lines.  

\vspace{0.07in}

\input{fig8_010}

Suppose that the additive weights are $b_1,\dots, b_n$, after the crossings and switches to upward orientation on additive strands are done. Then  the invariant 
\[
\jmath(\gamma_0)= \sum_{k=1}^{n-1}\langle\sum_{i=1}^k b_i,b_{k+1}\rangle. 
\]
A similar diagram and computation gives the formula 
\[
\jmath(\gamma_1) = - \sum_{k=1}^{m-1}\langle\sum_{i=1}^k b_i',b_{k+1}'\rangle,
\]
where $b_1',\dots, b_m'$ are the weights of $Z_1$ upon moving all additive points to the leftmost region. The diagram for $\gamma_1$ is given by reflecting a suitable diagram as in Figure~\ref{fig8_010} about the $x$-axis and reversing orientations of all additive lines. 
Finally, $\jmath(\gamma)=\jmath(\gamma_0)+\jmath(\gamma_1)$. 

\begin{prop}
    Pivotal monoidal category $\Catone$ is equivalent to the category $\mcC_G$ for $G=\Aff_1(\kk)$ as defined in Section~\ref{subsec_monoidal}. 
\end{prop}
\begin{proof}
    This is immediate from our results on $\Catone$. 
\end{proof}

The interesting part of the construction of $\Catone$ is the invariant $\jmath$ of morphisms related to the 2-cocycle $\brak{a,b}$. 


\subsection{Monoidal category \texorpdfstring{$\Cattwo$}{Cattwo}}\label{subsec_cattwo}

$\quad$

We now extend $\Catone$ to the monoidal category $\Cattwo$ by keeping the same objects as in $\Catone$ by adding dots labelled by elements of $J(\kk)$ floating in regions to the morphism diagrams, see Section~\ref{subsection:Cathelineau_vect_space}. First, we consider category $\CattwoR$ mentioned earlier and introduce a function $\jmath$ evaluating a morphism in $\CattwoR$ to an element of $J(\kk)$. A morphism $\gamma'$ in $\CattwoR$ is given by a morphism $\gamma$ in $\CatoneR$ together with a collection of $J(\kk)$-labelled dots in regions of $\gamma$. Denote by $D(\gamma')$ the set of dots of $\gamma'$. For each dot $d$ of $\gamma'$ let $\ell(d)\in J(\kk)$ be its label and recall from the discussion of \eqref{eq_func_j} the function $\omega(d,\gamma)\in \kk^{\times}$ that to a point $d$ of $\gamma$ not on a multiplicative line assigns its winding index relative to the leftmost region. Define 
\begin{equation}\label{eq_jmath_prime}
    \jmath(\gamma') \ := \ \jmath(\gamma) + \sum_{d\in D(\gamma')} \omega(d,\gamma)\ell(d) \ \in \ J(\kk),
\end{equation}
where $\jmath(\gamma)$ is given by \eqref{eq_func_j} and $\omega(d,\gamma)\ell(d)$  is the element $\ell(d)$ of $J(\kk)$ scaled by $\omega(d,\gamma)$. 

For example, consider the morphism $\gamma'$ in $\CattwoR$ in Figure~\ref{fig4_018b} extending the morphism $\gamma$ in Figure~\ref{fig4_018} by three dots. These dots have winding numbers $1,ec, e$ relative to the leftmost region of the diagram. 
Consequently, 
\begin{equation}\label{eq_jmath_prime_example}
\jmath(\gamma')=\jmath(\gamma)+ x + ec\cdot y + e\cdot z = e \brak{ac,b}+ x + ec\cdot y + e\cdot z 
\end{equation}

\input{fig4_018b}

Define the category $\Cattwo$ as the quotient of $\CattwoR$ by the relation: two morphisms $\gamma_1,\gamma_2:Z_0\lra Z_1$ in $\CattwoR$ are equal in $\Cattwo$ if and only if  $\jmath(\gamma_1)=\jmath(\gamma_2)$. 
Category $\Cattwo$ shares many properties with $\Catone$ and contains $\Catone$ as a subcategory. 
Category $\Cattwo$ is monoidal, with the tensor product given by placing diagrams in parallel. Any morphism in $\Cattwo$ is invertible. There exists a morphism $\gamma:Z_0\lra Z_1$ in $\Cattwo$ if and only if $\jmath(Z_0)=\jmath(Z_1)$. 

Given a morphism $\gamma_0:Z_0\lra Z_1$, any morphism $\gamma_1:Z_0\lra Z_1$ has a unique presentation as $\gamma_0$ with an additional dot labelled by some $x\in J(\kk)$ in the leftmost region. In other words, $\Hom_{\Cattwo}(Z_0,Z_1)$, when nonempty, is naturally a $J(\kk)$-torsor, for the action given by placing dots in the leftmost region. 

Given a diagram $D$ of a morphism $\gamma$ in $\Cattwo$, it can be replaced by a diagram where all the dots are collected in the leftmost region, with the label of each dot twisted by its winding number relative to that region. The labels can then be added to replace several dots by a single dot.  

Factorization of morphisms in $\Catone$ in \eqref{eq_factor_Z} extends to those in $\Cattwo$, adding a dot with an arbitrary label in $J(\kk)$ in the leftmost region.  

Category $\Cattwo$ is similar to the category described in Section~\ref{subsec_monoidal_U}, where $G=\Aff_1(\kk)$ and representation $U=J(\kk)$ with the action of $\Aff_1(\kk)$ given earlier. Evaluation $\jmath$ is analogous to the 2-cocycle evaluation in Section~\ref{subsec_two_cocycle}.  

\vspace{0.07in}
 
\input{fig5_003}

By analogy with Section~\ref{subsec_two_cocycle}, we can introduce a ``blue line", the bottom boundary of a diagram, that can absorb part of the network describing a morphism, in exchange emitting elements of $J(\kk)$.
Figure~\ref{fig5_003} shows that absorbing an $(a,b)$-merge requires emitting $\brak{a,b}$-dot, while an additive crossing can be absorbed at no cost. Figure~\ref{fig5_005}  interprets commutativity of $\brak{a,b}$ via emission and absorption of a merge and an additive crossing. 

\input{fig5_005}

\input{fig5_004}

The two-cocycle property of $\brak{a,b}$ is demonstrated in Figure~\ref{fig5_004}, which is similar to the computation in Figure~\ref{fig3_032} except that in the present case $\R$ acts trivially on $J(\kk)$. 

\input{fig5_008}

\input{fig5_009}

Figure~\ref{fig5_008} shows that a crossing of an additive and a multiplicative line is absorbed at no cost. Likewise, Figure~\ref{fig5_009} says that a local  maximum or minimim of a red line can be absorbed at no cost, and likewise for a merge and orientation reversal of red (multiplicative) lines. Split of a multiplicative line into two lines can be absorbed at no cost as well. 

\begin{remark}
    If one prefers, additive $0$-lines can be removed from diagrams, see Figure~\ref{fig5_015}. 
\end{remark}

\input{fig5_015}

%
%

\section{Diagrammatics of entropy}
\label{subsec_diag_ent}

\subsection{A monoidal category for entropy}
Let us specialize $\kk=\R$. Real vector space $J(\R)$ is enormous, and we reduce the complexity by using the entropy function 
 $H:J(\R)\lra \R$ given on generators of $J(\R)$ by \eqref{eq_map_H} and \eqref{eq_symbol}. Start with the category $\CattwoR$, where elements of $J(\R)$ float in regions of morphism diagrams and for two objects $Z_0,Z_1$ with $w(Z_0)=w(Z_1)$ the space of homs $\Hom_{\CattwoR}(Z_0,Z_1)$ is a $J(\R)$-torsor. Recall the invariant $\jmath$ that takes a morphism in $\CattwoR$ to $J(\R)$ and compose it with the entropy function to get an invariant 
 \begin{eqnarray}\label{eq_jmath_H}
 & & \jmath_H \ : \Hom_{\CattwoR}(Z_0,Z_1)\lra \R, \\ 
 & & \jmath_H:=H\circ \jmath, \ \  \Hom_{\CattwoR}(Z_0,Z_1)\stackrel{\jmath}{\lra} J(\R)\stackrel{H}{\lra} \R. 
 \end{eqnarray} 

Take the quotient of the category $\CattwoR$ by identifying two morphisms $\gamma_0,\gamma_1:Z_0\lra Z_1$ iff $\jmath_H(\gamma_0)=\jmath_H(\gamma_1)$ and denote the resulting category $\CattwoH$. 

A related way to define $\CattwoH$ is to consider a version of the category $\CattwoR$ where dots are labelled by elements of $\R$ rather than $J(\R)$. Otherwise it is the same setup with a morphism $\gamma$ given by a pair of overlapping multiplicative and additive networks together with a finite set of dots $D(\gamma)$, disjoint from the multiplicative network, and labelled by elements of $\R$. For a morphism $\gamma$ in this category there is a well-defined evaluation, which can also be denoted $\jmath_H(\gamma)$, given by 
\begin{equation}\label{eq_func_j_H}
\jmath_H(\gamma) = \sum_{p\in \mathrm{add}(\gamma)} s(p) \omega(p,\gamma)\langle a_p,b_p\rangle_H +\sum_{d\in D(\gamma)} \omega(d,\gamma)\ell(d) \ \in \ \R.
\end{equation}
This formula is analogous to  \eqref{eq_func_j} and \eqref{eq_jmath_prime}, with $\brak{a,b}_H$ given by \eqref{eq_symbol}. We sum over vertices $p$ of the additive network, multiplying the value of the symbol $\brak{a_b,b_p}_H\in \R$ by the winding number $\omega(p,\omega)\in \R^{\times}$ and the sign. 
In addition, sum over dots, scaling the label $\ell(d)$ of a dot $\ell$ by its winding number.

Category $\CattwoH$ is monoidal, with the same objects as in categories $\CatoneR$ and $\CattwoR$. There is a morphism $\gamma:Z_0\lra Z_1$ iff $w(Z_0)=w(Z_1)$. In the latter case, the set of homs from $Z_0$ to $Z_1$ is an $\R$-torsor.  For a morphism $\gamma_0:Z_0\lra Z_1$, any morphism $\gamma_1:Z_0\lra Z_1$ has a unique presentation as $\gamma_0$ with an extra dot labelled by some $x\in \R$ in the leftmost region. Furthermore, all dots in a diagram can 
be moved to the leftmost region, twisted by winding numbers, and added into a single dot. 
Factorization of morphisms works in the same way as for categories $\Catone$ and $\Cattwo$. 

Define $\CatoneH$ to be the subcategory of $\CattwoH$ consisting of dotless diagrams. In $\CatoneH$ there is a unique morphism $Z_0\lra Z_1$ iff $w(Z_0)=w(Z_1)$, and $\CatoneH$ is equivalent to the category $\mcC_G$ in Section~\ref{subsec_monoidal} for $G=\Aff_1(\R)$. 

Entropy is incorporated in the symbols $\brak{a,b}_H$ via \eqref{eq_symbol}. Entropy, when viewed as the function $\jmath_H(\gamma)$ on morphisms $\gamma$ in $\CatoneH$, depends only the source and target objects of the morphisms $\gamma$. In this sense, it may be viewed as a one-dimensional topological theory with defects, which are endpoints of additive and multiplicative lines on the boundary of a network (corresponding to objects $X^{\pm}_a$ and $Y^{\pm}_c$).

\begin{remark}\label{remark_extend} The extension of the 2-cocycle $\brak{a,b}_H$ from $\R$ to $\Aff_1(\R)$ in \eqref{eq_two_coc_o} can be motivated by Figure~\ref{fig5_012} representation of the multiplication in $\Aff_1(\R)$ (and, more generally, in $\Aff_1(\kk)$).  
For the extension considered in the present paper we make only the additive vertex in that diagram to contribute to the cocycle, and one extends the 2-cocycle from $\R$ to $\Aff_1(\R)$ via 
\begin{equation}\label{eq_two_coc}
    \langle (a_1,c_1),(a_2,c_2)\rangle_H \ := \ \langle a_1, c_1 a_2\rangle_H, 
\end{equation}
and likewise replacing $\R$ by $\kk$. 
That it is a (normalized)  2-cocycle on $\Aff_1(\R)$ can be encoded by the skein relation in Figure~\ref{fig8_004}. The same relation encodes the associativity of the multiplication in $\Aff_1(\R)$. This relation follows from the diagrammatic relations listed earlier. 

\input{fig8_004}

\end{remark} 

In Section~\ref{subsec_cattwo} we described category $\Cattwo$ and a way to evaluate a morphism in $\Cattwo$ by using the bottom boundary line of a diagram to absorb parts of a diagram while emitting elements of $J(\kk)$. 
Essentially the same construction applies to evaluate and manipulate morphisms in $\CattwoH$, replacing $\brak{a,b}$ in Figures~\ref{fig5_003},~\ref{fig5_005} and~\ref{fig5_004} by $\brak{a,b}_H\in R$. 

\vspace{0.07in} 

\input{fig8_002}

The network in Figure~\ref{fig8_002} on the left is an example of a morphism in $\CattwoH$ from the object $X^+_{p_1}\otimes \dots \otimes X^+_{p_n}$ to $X^+_1$. When the bottom boundary absorbs this network, it produces several dots that add up to the entropy $H(X)$ of a random variable $X$. Here $X$ takes some values $x_1,\dots, x_n$ with probabilities $p_1,\dots, p_n$, with $p_1+\ldots +p_n=1$, and we may also write $X=(p_1,\dots, p_n)$. 

A special case of this relation is shown in  Figure~\ref{fig5_101} on the left. 

\input{fig5_101}

\vspace{0.07in} 

\input{fig8_007}

Functional equation~\eqref{eq_H} for the entropy and its variations~\eqref{eq_H_symm},~\eqref{eq_H_modified} can be interpreted diagrammatically as equalities of evaluations of the following diagrams. Consider the transformation labelled ``associativity'' in Figure~\ref{fig8_007} from the diagram on the right to the middle. In can indeed be viewed as a version of the associativity for the addition, with one orientation of one of the edges reversed, see Figures~\ref{fig3_011} and~\ref{fig3_010} for similar transformations in an arbitrary group $G$ rather than in $\R$. 

 The merge and the split in the leftmost diagram in Figure~\ref{fig8_007} evaluate to $\langle p,1-p\rangle_H - \langle q,1-q\rangle_H = H(p)-H(q)$, The middle diagram, obtained via the associativity transformation from the one on the left evaluates to $-\langle q, p-q\rangle_H + \langle p-q, 1-p\rangle_H 
= -p H(\frac{q}{p})+(1-q)H(\frac{p-q}{1-q})$. Inserting weight $0$ line, whose endpoints contribute $0$, as shown  in the rightmost diagram, results in the evaluation  
$-\langle q,p-q\rangle_H 
+ \langle q-1,p-q\rangle_H + \langle p-1,1-p\rangle_H -\langle q-1, 1-q\rangle_H = -p H(\frac{q}{p}) 
+ (p-1)H(\frac{q-1}{p-1}) + 0 - 0$. This gives the 4 term relation in \eqref{eq_H}.

\input{fig8_008}

\input{fig8_009}

Transformation of diagrams in Figure~\ref{fig8_008} corresponds to the relation $\brak{a,b}_H=-\brak{-a-b,a}_H$ (modulo the relations $\brak{u,-u}_H=0$). It also corresponds to the symmetry property \eqref{eq_H_5} of the entropy (modulo \eqref{eq_H_4}, which corresponds to the transformation in Figure~\ref{fig5_002} top right). 

These symmetries of the bracket $\brak{a,b}_H$ and the entropy $H(p)$ are described by an action of $S_3$ on the space of parameters $p\in \R$ as shown in Figure~\ref{fig8_009}. 

\vspace{0.07in} 

{\bf Chain rule for the entropy.}
The entropy function for finite random variables can be described via the formula \eqref{eq_H_2} for random variables that take two values and the chain rule~\cite{Faddeev56,Faddeev_online,Khinchin56,BFL11}.  

\input{fig8_003}

Figure~\ref{fig8_003} interprets diagrammatically the simplest instance of the chain rule. Label $Y$ near a vertex denotes a random variable $(p,1-p)$, with the line going out of the corresponding vertex labelled 1. The line gets rescaled by the $1-c$ red line. The other additive vertex on the LHS represents a random variable $Z=(c,1-c)$. Diagram on the right corresponds to the random variable $X=(c,p(1-c),(1-p)(1-c))$. The equality of evaluations in Figure~\ref{fig8_005} is the simplest instance of the chain rule: 
\begin{equation}\label{eq_chain_simplest}
H(X)=H(Z)+(1-c)H(Y).
\end{equation}

\input{fig8_005}

This example can be generalized in various ways. Consider the LHS of Figure~\ref{fig8_006}  encoding finite random variables $Y_1,\dots, Y_n$ with $Y_i=(p_{1,i},\dots, p_{k_i,i})$ and a random variable $Z=(p_1,\dots, p_n)$. Each variable $Y_i$ is represented by a collection of additive lines that merge into an vertex located below the corresponding multiplicative (wavy) line. The lines merge into a line that carries label $1$. Upon crossing the $i$-th multiplicative line, of weight $p_i$, the additive line get scaled to $p_i$. Next, additive lines $p_1,\dots, p_n$ merge into additive line $1$, describing a random variable $Z$. This merge (and the merges for $Y_i$'s) is shown as a single merge, but it can be more accurately defined as a composition of pairs of lines merging and adding their weights, see Figure~\ref{fig8_002} on the left, for instance. 

 On the RHS of Figure~\ref{fig8_006} additive lines at the bottom of the diagram first cross multiplicative lines, then merge in groups, then eventually merge into a single line $1$. This composition of merges is denoted by a single variable 
\[
X=(p_1p_{1,1},\dots, p_1p_{k_1,1},p_2p_{1,2},\dots, p_2p_{k_2,2},\dots, p_np_{k_n,1}, \dots,p_np_{k_n,n}).
\]

\input{fig8_006}

\vspace{0.07in}

 To go from the left to the right figure we  pull the red lines on the LHS of Figure~\ref{fig8_006} downward until they are below corresponding additive vertices. The LHS diagram evaluates to the sum of the entropy of $Z$ and entropies of $Y_i$ scaled by $p_i$, due to the presence of multiplicative lines. 

The equality of the two evaluations is then 
the chain rule for the entropy: 
\begin{equation}\label{eq_ent_chain} 
H(X) \ = \ H(Z)+ \sum_{i=1}^n p_i H(Y_i) . 
\end{equation} 

{\bf Extending the category of finite probabilistic spaces.}
Category $\CattwoH$ contains a  subcategory equivalent to the Baez-Fritz-Leinster category $\FinProb$ as defined in~\cite{BFL11}. Objects of $\FinProb$ are finite sets equipped with probability measures $(X,p)$ and morphisms are measure-preserving functions on finite sets $f:(X,p)\lra (Y,q)$ with $q_j=\sum_{i\in f^{-1}(j)}p_i$. Any morphism $f:X\to Y$ is surjective onto points of $Y$ of strictly positive measure. 

Consider the subcategory $\BFL$ of $\CattwoH$ with  objects given by tensor products 
\[
X^+_{\underline{p}}:=X^+_{p_1}\otimes \cdots \otimes X^+_{p_n}, \ \ 0\le p_i\le 1, \ \sum_{i=1}^n p_i=1, \ \underline{p}:=(p_1,\dots, p_n).
\]
Morphisms are generated by merges $X^+_{p}\otimes X^+_{p'}\lra X^+_{p+p'}$, for $p+p'\le 1$  and transpositions $X^+_{p}\otimes X^+_{p'}\lra X^+_{p'}\otimes X^+_{p}$ (see Section~\ref{subsec_mon_for_dilog} for a list of generating morphisms). 

These are exactly dotless diagrams made only of additive upward-oriented lines of non-negative weight $p$ with $0\le p\le 1$ with no splits, with the total weight $1$ in any cross-section by a horizontal line. Equivalently, these are upward-oriented additive line diagrams with merges and permutations of strands only, total weight $1$, and weight restriction $0\le p\le 1$ for each line.  

 Subcategory $\BFL$ is equivalent to $\FinProb$. An equivalence is given by the functor $F:\BFL\lra \FinProb$ that takes  $X^+_{\underline{p}}$ to a finite $n$-point set with probabilities $p_1,\dots, p_n$ assigned to the points. It is clear how to extend the functor to morphisms. Notice that $\BFL$ is also a subcategory in $\CatoneH\subset \CattwoH$, since diagrams describing morphisms in $\BFL$ are dotless. 

Functoriality property of entropy in $\FinProb$ corresponds to additivity of the entropy on the composition of morphisms, including in the larger category $\CattwoH$, that is, $\jmath(\gamma'\gamma)=\jmath(\gamma')+\jmath(\gamma)$. 
Convex linearity of entropy in $\FinProb$, see~\cite{BFL11}, says that 
\[
H(\lambda Y_1\oplus (1-\lambda)Y_2)=\lambda H(Y_1)+(1-\lambda)H(Y_2),
\]
where random variable $\lambda Y_1\oplus (1-\lambda)Y_2$ is associated to random variables $Y_1,Y_2$ and $\lambda, 0<\lambda<1$ in the obvious way. Convex linearity corresponds to the equality of evaluations in Figure~\ref{fig8_006} when $n=2$ and $\lambda=p_1$. 

\vspace{0.07in}

The category $\BFL$, while equivalent to the Baez--Fritz--Leinster category, is a small part of $\CatoneH$. It is natural to wonder whether the bigger categories $\CatoneH$ and $\CattwoH$, closely related to the affine group $\Aff_1(\R)$, may be relevant to probability or whether some modification of $\CatoneH$ and $\CattwoH$ could extend probability as it appears in the quantum world. We refer to~\cite{BFL11,Brad21,Lein21_entro_div,Gromov13,Vign20,BB15} and follow-up papers for more on the categorical meaning of entropy. Operadic interpretation of entropy by Bradley~\cite{Brad21} and Leinster~\cite{Lein21_entro_div} is closely related to the diagrams in Figure~\ref{fig8_006}. 

In this section we offered a different categorical interpretation of entropy of a finite probability space from what exists in the literature. Our interpretation does not have an immediate application to probability. Still, we hope that the present paper complements existing work in a useful way. 


\subsection{More entropy and cocycle remarks}\label{subsec_more_ent}

\begin{remark} 
    Consider the action of $\Aff_1(\R)$ on the abelian group $(\R,+)$ via $(a,c)u=c |c|^{\alpha-1}u$, where $\alpha\in \R_{>0}\setminus 1$. 
    Note that $ c|c|^{\alpha-1}=\mathrm{sign}(c) |c|^{\alpha}$. 
    One can recover the Tsallis entropy by looking at a suitable 2-cocycle on $\Aff_1(\R)$.
    Tsallis entropy is given by 
    \[
    H_{\alpha}(X)=\frac{1}{\alpha-1}\biggl( 1-\sum_{i\in X}p_i^{\alpha}\biggr) . 
    \]
    Interpreting $1=(\sum_{i\in X}p_i)^{\alpha}$ we can define
    \[ \psi_{\alpha}(x)= \begin{cases} x|x|^{\alpha-1}, & \mathrm{if} \ x \in \R^{\ast},\\
     0 & \mathrm{if} \ x=0
    \end{cases}
    \]
    and 
    \[\brak{a,b}_{\alpha}\ :=\ \psi_{\alpha}(a)+\psi_{\alpha}(b)-\psi_{\alpha}(a+b). 
    \]
    Observe that $\psi_{\alpha}:\R\lra\R$ is a homomorphism. 
    Then 
    \[
    (\alpha-1)H_{\alpha}(p)= \pm \brak{p,1-p}_{\alpha}.
    \]
    Symmetry property holds: $\brak{b,a}_{\alpha}=\brak{a,b}_{\alpha}$. For the scaling property we have $\psi_{\alpha}(ca)=\psi_{\alpha}(c)\psi_{\alpha}(a)$ and 
    \[
    \brak{ca,cb}_{\alpha}= \psi_{\alpha}(c) \brak{a,b}_{\alpha}. 
    \]
    We extend the cocycle to the entire $\Aff_1(\R)$ group via 
    \begin{equation}\label{eq_two_coc_tsallis}
    \langle (a_1,c_1),(a_2,c_2)\rangle_{\alpha} \ := \ \langle a_1, c_1 a_2\rangle_{\alpha}
\end{equation}

In the diagrammatics for the analogue of category $\mcC_{H}^{\circ}$ in this setup, an additive line $a$ crossing a multiplicative line $c$ co-oriented away from it still gets rescaled to $ca$. But a dot labelled $x\in \R$ moving through a multiplicative $c$-line co-oriented away from the dot is scaled to $\psi_{\alpha}(c)x$. In the evaluation function $\jmath_{\alpha}$ the 2-cocycle $\brak{a,b}$ in \eqref{eq_func_j} and the 2-cocycle $\brak{a,b}_H$ in \eqref{eq_symbol} should be replaced by $\brak{a,b}_{\alpha}$ and the action of $\R^{\times} $ on $\R$ by the action via $\psi_{\alpha}$, so that $c \dot x= \psi_{\alpha}(c)x$: 
\begin{equation}\label{eq_func_j_alpha}
\jmath_{\alpha}(\gamma) = \sum_{p\in \mathrm{add}(\gamma)} s(p) \psi_{\alpha}(\omega(p,\gamma))\langle a_p,b_p\rangle_{\alpha} \ \in \ \R.
\end{equation}
and, by analogy with \eqref{eq_jmath_prime},
\begin{equation}\label{eq_jmath_prime_alpha}
    \jmath_{\alpha}(\gamma')  =  \jmath_{\alpha}(\gamma) + \sum_{d\in \gamma'} \psi_{\alpha}(\omega(d,\gamma)\ell(d)) \ \in \ \R
\end{equation}
in the presence of $\R$-weighted dots. 

A cohomological interpretation of the Tsallis entropy is discussed in Vigneaux~\cite{Vign20}.  Another interesting problem is to look for cohomological interpretation of other entropies, including the R\'enyi entropy. P.~Tempesta~\cite{Temp16} has discovered a relation between entropies and formal groups, and it is a natural question whether this relation can be given a cohomological meaning. 
\end{remark}

\begin{remark}
      Recall a non-measurable entropy-like function in \eqref{eq_bad_entropy} given by 
     \[
     h(x)=((dx)^2)/(x(1-x)) \ \mathrm{for}\ x\in\R, \ x\not= 0,1  \ \ \mathrm{and}\  \ h(0)=h(1)=0,
     \]
     where $d:\R\lra \R$ is a $\Q$-linear derivation. Noting that $1/(x(1-x))=1/x+1/(1-x)$, introduce a function $\psi:\R\lra \R$ given by 
     \begin{equation}\label{eq_bad_psi}
     \psi(a)=\begin{cases} (da)^2/a & \mathrm{if}\ a\not=0,1, \\
     0 & \mathrm{if} \ a=0,1. 
     \end{cases} 
     \end{equation}
     Then $h(x)=\psi(x)+\psi(1-x)$. Function $\psi$ is not a non-linear derivation, unlike~\eqref{eq_psi_derivation}, since 
     $\psi(ab)\not=\psi(a)b+a\psi(b)$. 
Instead, we compute 
\[
d(a/(a+b)) = ((da)(a+b)-a(da+db))/(a+b)^2=
((da)b - a(db))/(a+b)^2
\]
and define 
\begin{eqnarray*}
\brak{a,b} & := & (a+b)h(a/(a+b))=(a+b)^3(d(a/(a+b))^2)/(ab), \\
 & = &(a+b)^{-1}  ((da)b - a(db))^2/ab  
 \\
 & = & (da)^2 b/(a(b+a)) + (db)^2 a/(b(a+b)) - 2 (da)(db)/(a+b) 
\end{eqnarray*}
Next, 
\begin{eqnarray*}
 \psi(a)+\psi(b)-\psi(a+b) & = & (da)^2/a+(db)^2/b - (da+db)^2/(a+b)  \\
 & = & (b^2 (da)^2 + a^2 (db)^2 - 2ab(da)(db))/(ab(a+b)) 
   = \brak{a,b}
\end{eqnarray*}
Consequently, 
\begin{equation}
    \brak{a,b} \ = \ \psi(a)+\psi(b)-\psi(a+b), 
\end{equation}
and the 2-cocycle $\brak{a,b}$ on $\R$ is the coboundary of $\psi$. The bracket is symmetric, $\brak{a,b}=\brak{b,a}$, and a direct computation shows the scaling property  
\[
\brak{ca,cb} = c\brak{a,b}.  
\]
On the field $\R$, function $h(x)$ and the 2-cocycle $\brak{a,b}$ are not measurable and cannot be written down explicitly. Instead, start with a base field $\kk$ and  form the field $K=\kk(y)$, where $y$ is a formal variable. Pick $r=r(y)\in K$ and define $dy=r$,
$d(f/g)=(d(f)g-f d(g))/g^2$, for polynomials $f,g\in \kk[y]$, resulting in a derivation $d:K\lra K$. Via formula \eqref{eq_bad_psi} we get a null-homologous $\kk$-linear 2-cocycle $\brak{a,b}:K\times K\lra K$. 

This cocycle extends to a 2-cocycle on the group $\Aff_1(K)$ by 
\begin{equation}\label{eq_two_coc_two}
    \langle (a_1,c_1),(a_2,c_2)\rangle \ := \ \langle a_1, c_1 a_2\rangle, \ \ a_i\in K, c_i\in K^{\ast}, 
\end{equation}
as in~\eqref{eq_two_coc}. We do not classify $r$'s that give rise to non-homologous 2-cocycles in this case but refer to Cathelineau's result~\cite[Corollary 4.9]{Cath11} that the space $\mathrm{H}_2(\Aff_1(K),K_r)$ is infinite-dimensional over $K$. 
\end{remark}

\begin{remark}  
The \emph{unrolled quantum group} $\mathfrak{sl}(2)$, see~\cite{CGP17,CGP17_corr}, 
contains both elements $H$ and $K=q^H$. The trace of the element $HK$ on a representation of  the modified quantum $\mathfrak{sl}(2)$ has the flavor of the von Neumann entropy, being the sum $\sum_i \lambda_i  \log(\lambda_i)$, where $\lambda_i$'s are the eigenvalues of $K$.
\end{remark}

\begin{remark}
{\it Generalized binomials as 2-cocycles.} The binomial coefficient $c(k_1,k_2):=\binom{k_1+k_2}{k_1}=\frac{(k_1+k_2)!}{k_1!k_2!}$ satisfies the 2-cocycle equation 
\[
c(k_1,k_2)c(k_1+k_2,k_3)=c(k_2,k_3)c(k_1,k_2+k_3).
\]
Consequently, it can be viewed as a 2-cocycle on the monoid $(\Z_+,+)$ of non-negative integers under addition, taking values in the abelian group $\Q^{\ast}$. The binomial cocycle is null-homologous, since $c(k_1,k_2)=f(k_1+k_2)f(k_1)^{-1}f(k_2)^{-1}$. It is the coboundary of 1-cochain $f:\Z_+\lra \Q^{\ast}$ given by $f(k)=k!$.  

Likewise, the q-binomial coefficient $c_q(k_1,k_2):=\frac{[k_1+k_2]!}{[k_1]![k_2]!}$, where $[k]!:=[k]\cdots [2][1]$ and $[n]:=\frac{q^n-q^{-n}}{q-q^{-1}}=q^{n-1}+q^{n-3}+\ldots + q^{1-n}$ is a 2-cocycle $c_q:\Z_+\times \Z_+ \lra \Q(q)^{\ast}$, valued in the group of invertible elements of the field of rational functions in $q$ with rational coefficients. Again, this 2-cocycle is null-homologous, being the coboundary of the 1-cochain $f:\Z_+\lra \Q^{\ast}$, $f(k)=[k]!$.  

Binomials and generalized binomials are interpreted cohomologically in~\cite{Vig23} in a more sophisticated way, and we are not aware of a reference or use for this trivial interpretation, as null-homologous 2-cocycles on the monoid $\Z_+$ with values in $\Q^{\ast}$.  

More generally, pick a function $g:\Z_+\lra \kk^{\times}$, where $\kk$ is a field, and define a 1-cochain $f:\Z_+\lra \kk^{\ast}$ by $f(0)=1$ and $f(k)=g(1)g(2)\cdots g(k)$ for $k>0$. Then the generalized binomial coefficient (the Fonten\'e--Ward coefficient~\cite{Vig23}) 
\[
c_g(k_1,k_2):=f(k_1+k_2)f(k_1)^{-1}f(k_2)^{-1}
\]
is the coboundary of $f$. 

\vspace{0.07in}

Note that~\cite{Asymptotics20,Garou09} for $a>b>0$ 
\[
a H\left(\frac{b}{a}\right) = \lim_{n\to\infty} \frac{1}{n}\log\biggl( \binom{an}{bn} \biggr), 
\]
realizing the entropy as a limit via binomial coefficients (set $a=1$ to get $H(b)$). 

\vspace{0.07in}

The factorial function extends to the Gamma function $\Gamma:\R_{>0}\lra \R_{>0}$ by 
\[
\Gamma(x) \ :=  \ \int_0^{\infty} t^{x-1}e^{-t}dt,
\]
with $n!=\Gamma(n+1)$. 
Beta function 
\[
B(x_1,x_2)  := \frac{\Gamma(x_1)\Gamma(x_2)}{\Gamma(x_1+x_2)}= \int_0^1 t^{x_1-1}(1-t)^{x_2-1}dt 
\]
gives rise to a null-homologous 2-cocycle 
\[
\beta:\R_+\times \R_+\lra \R^{\ast}, \ \ 
\beta(x_1,x_2)=B(x_1+1,x_2+1).
\]
Monoid $\R_+$ can be replaced by the larger monoid $\C_+=\{z\in \C|\textsf{Re } z\ge 0\}$, under addition, which still avoids the poles of $\Gamma(z+1)$, resulting in a null-homologous 2-cocycle $\beta:\C_+\times \C_+\lra \C^{\ast}$. Poles of $\Gamma(z)$ provide an obstruction to extending the 2-cocycle $\beta$ to all of $\C$. 
\end{remark}

\begin{remark} {\it The PMI cocycle.} Pointwise Mutual Information (PMI) between a pair of discrete outcomes $x$ and $y$ of random variables $X$ and $Y$ is given by 
\begin{equation}
    \PMI(x,y)  \ = \ \log \frac{p(x,y)}{p(x)p(y)}. 
\end{equation}
Here it makes sense at first to restrict to $p(x),p(y)\not= 0$, to avoid a possible zero in the denominator. 
Even with that restriction, $\PMI(x,y)$ may take value $-\infty$, if $x$ and $y$ are disjoint, that is, $p(x,y)=0$. Thus, $\PMI(x,y)\in \R\cup \{ -\infty\}$. 

Pointwise Mutual Information is commonly used in linguistics and natural language processing (NLP)~\cite{Pointwise_mutual_info,Fano68,CH90,Bouma09}. The original word embedding construction of Mikolov et al.~\cite{MCCD13,MSCCD13} is interpreted as factorizations of the PMI matrix and its variations in Levy and Goldberg~\cite{LG14}. 

Observe that $\PMI$ is a null-homologous 2-cocycle, since 
\[
\PMI(x,y) = \log p(x,y) - \log p(x) - \log p(y) = \log p(x \cap y) - \log p(x) - \log p(y). 
\]
Here $x,y$ are be viewed as elements of the  commutative monoid $\mathsf{M}(\Omega)$ of measurable subsets of a probability space $\Omega$, with $x\sim x'$ if $p(x\Delta x')=0$ and $x+y:=x\cap y$. Random variables $X$ and $Y$ above are discrete measurable functions on $\Omega$. 
Map 
\begin{equation}
\PMI \  : \ \mathsf{M}(\Omega)\times \mathsf{M}(\Omega) \lra \widetilde{\R}, \ \ \widetilde{\R}:=\R \cup \{-\infty\} 
\end{equation}
is a null-homologous 2-cocycle on the commutative monoid $\mathsf{M}(\Omega)$ taking values in the commutative monoid $\widetilde{\R}$. For consistency, we can set $\PMI(\emptyset, y)=-\log p(y)$.  Nevertheless, one runs into the issue that subtraction is not well-defined in $\widetilde{\R}$ when claiming that cocycle $\PMI$ is null-homologous, see a related short discussion close to the end of Remark~\ref{rmk_carry} below. 
\end{remark} 

\begin{remark} {\it Penrose impossibility cocycle.}
    R.~Penrose~\cite{Pen93} gave an amusing proof of the impossibility of the Escher \emph{tribar} shape via the first cohomology group of a non-simply connected plane region, with coefficients in $\R$. It is an interesting question whether such methods may have a practical use, for instance, for detection of computer and AI generated images and videos.  

We discovered this reference by browsing through online threads~\cite{Cor12,Gia17} dedicated to appearances of cocycles and cohomology and refer to these threads for more examples. 
\end{remark}

\begin{remark}\label{rmk_carry}
{\it The carry cocycle.}
When adding two numbers represented in base $N$ there is the carry map on digits $\Z/N\times \Z/N\lra \{0,1\}\subset \Z/N$ given by $(i,j)\mapsto \lfloor \frac{i+j}{N}\rfloor$. It has been pointed out and rediscovered several times~\cite{Dol23,Isa02,DGPV23} that the map is a 2-cocycle which defines a nontrivial extension $0\lra \Z/N\lra \Z/(N^2)\lra \Z/N$ of abelian groups. This 2-cocycle is \emph{not} null-homologous and gives a generator of the cohomology group $\mathsf{H}^2(\Z/N,\Z/N)$. 

The carry $\lfloor \frac{i+j}{N}\rfloor$ can be thought of as \emph{excess} that digits $i$ and $j$ produce when added. Upon reflection, one realizes that many operations with ``excess'' in them can be interpreted as 2-cocycles. Let us provide some examples. 

(a) Consider a set $\mathcal{S}$ of companies. Denote by $p(A)$ yearly profit of a company $A\in \mathcal{S}$. Suppose that if two companies merge, they can streamline operations and produce excess profit $e(A,B)$ per year. Denote by $A+B$ the merged company. Then the excess profit 
\begin{equation}\label{eq_excess_profit}
e(A,B)=p(A+B)-p(A)-p(B).
\end{equation} 
The excess profit is a symmetric 2-cocycle: $A,B,C$ can merge into $A+B+C$ in different orders, with 
\[
p(A,B) + p(A+B,C) = p(B,C)+p(A,B+C),  
\]
where both sides show the total excess profit of going from $A,B,C$ to $A+B+C$. 
The 2-cocycle equation also follows at once from \eqref{eq_excess_profit}. 
Equation \eqref{eq_excess_profit} tells us that 2-cocycle $p$ is null-homologous: $e=dp$, the differential of the one-cochain $p$. 
To phrase this carefully, consider the Boolean monoid $\mathcal{PS}$ of all subsets of $\mathcal{S}$, with addition being the disjoint union. In particular, $A+A=A$ (if a company $A$ merges with itself, it remains $A$). Then $p:\mathcal{PS}\lra \R$ is a one-cochain (a function) from the abelian idempotent monoid $\mathcal{PS}$ to the abelian group of real numbers. Here, for a subset $\mcU\subset \mathcal{S}$, we define $p(\mcU)$ as the profit of the company given by merging all companies in $\mcU$. Furthermore, $e=dp$ is a null-homologous $\R$-valued 2-cocycle, 
\[
e \ : \ \mathcal{PS}\times \mathcal{PS} \lra \R. 
\]
One can further allow a company to split into two companies, going from $A$ to $(A_1,A_2)$, and extend $e$ to the split operation correspondingly. The diagrammatics of split and merge will be similar to the ones developed in the present paper for groups. One can also introduce red lines, describing external changes that modify profits $p(\mcU)$ of $\mcU\in \mathcal{PS}$. It is not clear whether such diagrammatics may be useful.  
\vspace{0.07in} 
   
(b) For a similar example, Let $\mathcal{S}$ be a set of people, $\mathcal{T}$ a set of tasks (or achievements) that people can accomplish. For a subset $\mcU\subset \mathcal{S}$ denote by $t(\mcU)\subset \mathcal{T}$ the set of tasks that the subset $\mcU$ of people can accomplish together. The one-cochain $t:\mathcal{PS}\lra \mathcal{PT}$ is a function from subsets of people to subsets of tasks. 

Consider excess tasks 
\begin{equation}\label{eq_e_tasks}
e(\mcU_1,\mcU_2) \ := \ t(\mcU_1+\mcU_2)\setminus (t(\mcU_1)+t(\mcU_2)). 
\end{equation} 
These are tasks that the union $\mcU_1+\mcU_2:=\mcU_1\cup \mcU_2$ of people can accomplish but that cannot be accomplished individually by $\mcU_1$ or $\mcU_2$. 
Excess tasks is a function 
\[
e \ : \ \mathcal{PS}\times \mathcal{PS}\lra \mathcal{PT}.
\]
Here both $\mathcal{PS}$ and $\mathcal{PT}$ are idempotent commutative monoids of subsets of $\mathcal{S}$ and $\mathcal{T}$, respectively. Equation \eqref{eq_e_tasks} indicates that $e=dt$, the coboundary of the one-cochain $t$. One should be careful here, since $\mathcal{PT}$, the target of maps $t$ and $e$, is not an abelian group, only a monoid, but with an additional operation $\setminus$. Defining cohomology structures in this setup might take one into nonabelian homological algebra~\cite{Grand13,CC19}.

\vspace{0.07in} 

In the above two examples the 2-cocycles are null-homologous. Null-homologous 2-cocycles are less interesting that not null-homologous ones, and it would make sense to look for more examples similar to the carry cocycle, representing nontrivial cohomology classes. 
\end{remark}

\begin{remark} {\it Affine group acting on a bank account.} Let us view the amount of money $x$ one has in a bank account as an element of the abelian group $M=\Q$ (for instance, one may have $\$100.5$ in the account), assuming that the bank  allows negative amounts. One may add or subtract some amount (debits and credits), sending $x$ to $x+b$, for $b\in \Q$. The account may also earn interest, scaling $x$ to $ax$, $a\in \Q^{\ast}$ within a unit time (typically, $a-1$ is a small positive number). 

    We see that the affine group $\Aff_1(\Q)$ acts on the set of possible values $M=\Q$ that money in a bank account can take. The action is not linear but can be made linear by embedding $M$ as an affine line in $\Q^2$ via the map $x\longmapsto (x,1)$, with the action
    \begin{equation}
        \begin{pmatrix}
a & b \\ 0 & 1
\end{pmatrix} \, \begin{pmatrix} x \\ 1\end{pmatrix} \ = \ \begin{pmatrix} ax+b \\ 1
\end{pmatrix}. 
    \end{equation}

    Can this toy example be extended to a nontrivial structure relevant to economics or finance? For instance, would it make sense to replace vectors $(x,1)$ by arbitrary vectors $(x,y)\in \Q^2$ that are acted upon by a subgroup of $\mathsf{GL}(2,\Q)$ other than $\Aff_1(\Q)$, via the usual matrix action 
    \begin{equation}
        \begin{pmatrix}
a_{11} & a_{12} \\ a_{21} & a_{22}
\end{pmatrix} \, \begin{pmatrix} x \\ y\end{pmatrix} \ = \ \begin{pmatrix} a_{11}x+a_{12}y \\ a_{21}x+a_{22}y
\end{pmatrix}? 
    \end{equation}
For another relation of the affine group to finance, note that Kelly's original derivation of his famous betting formula~\cite{Kelly56,Tho11} is based on entropy and information theory, in turn related to $\Aff_1(\R)$ as discussed earlier. 
\end{remark}

\bibliographystyle{amsalpha} 
\bibliography{z_brauer-group}

\end{document}